\documentclass[a4paper, 11pt, twoside]{article}
\usepackage{fullpage}

\usepackage{amsmath}
\usepackage{amssymb}
\usepackage{mathtools}
\usepackage{amsthm}
\usepackage{nicefrac}
\usepackage{multirow}
\usepackage{tikz-cd} 
\usepackage{subcaption}
\usepackage{multicol,caption}

\usepackage{wrapfig}
\usepackage{algorithm}

\usepackage[flushleft]{threeparttable} 
\usepackage{colortbl}
\definecolor{bgcolor}{rgb}{0.8,1,1}
\definecolor{bgcolor2}{rgb}{0.8,1,0.8}
\definecolor{niceblue}{rgb}{0.0,0.19,0.56}

\usepackage{hyperref}
\hypersetup{colorlinks,linkcolor={blue},citecolor={niceblue},urlcolor={blue}}

\usepackage{pifont}
\definecolor{PineGreen}{RGB}{0,110,51}
\definecolor{BrickRed}{RGB}{143,20,2}
\newcommand{\cmark}{{\color{PineGreen}\ding{51}}}%
\newcommand{\xmark}{{\color{BrickRed}\ding{55}}}%

\usepackage{thmtools}

\newcommand{\Exp}[1]{\mathbb{E} \left[ #1\right]}
\newcommand{\Expk}[1]{\mathbb{E}_k \left[ #1\right]}

\usepackage[colorinlistoftodos,bordercolor=orange,backgroundcolor=orange!20,linecolor=orange,textsize=scriptsize]{todonotes}

\definecolor{shadecolor}{gray}{0.9}
\declaretheoremstyle[
headfont=\normalfont\bfseries,
notefont=\mdseries, notebraces={(}{)},
bodyfont=\normalfont,
postheadspace=0.5em,
spaceabove=\topsep,
mdframed={
  skipabove=8pt,
  skipbelow=8pt,
  hidealllines=true,
  backgroundcolor={shadecolor},
  innerleftmargin=4pt,
  innerrightmargin=4pt}
]{shaded}

\declaretheorem[style=shaded,within=section]{definition}
\declaretheorem[style=shaded,sibling=definition]{theorem}
\declaretheorem[style=shaded,sibling=definition]{proposition}
\declaretheorem[style=shaded,sibling=definition]{assumption}
\declaretheorem[style=shaded,sibling=definition]{corollary}

\declaretheorem[style=shaded,sibling=definition]{lemma}

\usepackage[capitalize,noabbrev]{cleveref}

\usepackage{xspace}

\newcommand{\algname}[1]{{\color{PineGreen}\sf  #1}\xspace}

\newcommand{\R}{\mathbb{R}}

\newcommand{\cO}{{\cal O}}
\newcommand{\cS}{{\mathcal S}}

\newcommand{\hx}{\hat{x}}

\newcommand{\la}{\left\langle}

\newcommand{\ra}{\right\rangle}

\newcommand{\eqdef}{:=}
\newcommand{\argmin}{\text{argmin}}

\newcommand{\norm}[1]{\left\|#1\right\|}
\newcommand{\sqnorm}[1]{\left\| #1 \right\|^2}
\newcommand{\inprod}[2]{\left\langle #1,#2 \right\rangle}

\newcommand{\reals}{\mathbb{R}}
\newcommand{\cX}{\mathcal{X}}

\newcommand{\holderexponent}{\nu}
\newcommand{\cC}{\mathcal{C}}

\usepackage{amsmath,amsfonts,amssymb,amsthm,array}
\usepackage{caption}

\usepackage{graphicx}
\usepackage{booktabs} 

\usepackage{tabularx}
\usepackage{pifont}

\providecommand{\cmark}{\ding{51}}
\providecommand{\xmark}{\ding{55}}

\usepackage{float}
\usepackage{enumitem}
\usepackage{nicefrac}

\usepackage{algorithmic}

\usepackage{tikz}
\usetikzlibrary{calc,arrows.meta}

\newenvironment{tikzfigure}[1][htbp]
    {\begin{figure}[#1]\centering}
    {\end{figure}}

\theoremstyle{plain}

\newtheorem{rem}{Remark}[]

\theoremstyle{remark}

\usepackage{mdframed} 
\usepackage{thmtools}

\usepackage[flushleft]{threeparttable} 

\usepackage{subcaption}
\usepackage{multirow}
\usepackage{colortbl}
\definecolor{bgcolor}{rgb}{0.8,1,1}
\definecolor{bgcolor2}{rgb}{0.8,1,0.8}
\definecolor{niceblue}{rgb}{0.0,0.19,0.56}

\usepackage{hyperref}
\hypersetup{colorlinks,linkcolor={blue},citecolor={niceblue},urlcolor={blue}}

\usepackage[numbers]{natbib}
\usepackage{sidecap}

\definecolor{shadecolor}{gray}{0.9}
\declaretheoremstyle[
headfont=\normalfont\bfseries,
notefont=\mdseries, notebraces={(}{)},
bodyfont=\normalfont,
postheadspace=0.5em,
spaceabove=1pt,
mdframed={
  skipabove=8pt,
  skipbelow=8pt,
  hidealllines=true,
  backgroundcolor={shadecolor},
  innerleftmargin=4pt,
  innerrightmargin=4pt}
]{shaded}

\usepackage{xspace}

\begin{document}
\title{\bf Polyak-Type Extragradient Methods \\for Monotone Root-Finding Problems}

\author{TaeHo Yoon  \quad Sayantan Choudhury \quad Ezra Greenberg  \quad Nicolas Loizou}
\date{
Department of Applied Mathematics \& Statistics\\ 
Mathematical Institute for Data Science (MINDS)\\
Johns Hopkins University 
}    
\maketitle

\begin{abstract}
We study Polyak-type step-size selection for extragradient methods for solving deterministic and stochastic monotone root-finding problems. 
We show that the known projection-type correction for deterministic extragradient arises from minimizing an upper bound on the distance to a solution, paralleling the classical Polyak step-size construction. 
Using this viewpoint, we provide a unified deterministic analysis of the Polyak-type Extragradient Method (PolyakEG), based on a local critical condition controlling the variation of operator $F$ along the extrapolation direction. 
This analysis does not require global Lipschitz continuity, and covers sublinear convergence under broader conditions such as H\"older continuity or $(L_0, L_1)$-Lipschitzness and linear convergence under additional strong monotonicity, all through a single framework.
We then study the stochastic extensions of this approach. 
We first prove convergence of a direct stochastic variant, PolyakSEG, 
when all stochastic component operators share a common solution.
We also show that, without this condition, PolyakSEG with nonvanishing step-sizes may fail to converge to a zero of the mean operator. 
To address this limitation, we propose DecPolyakSEG, which combines decreasing step-sizes with Polyak-type updates, and establish a sublinear residual convergence result without requiring a common solution across the component operators.
These results parallel recent developments in stochastic Polyak step-sizes from the convex minimization literature and establish an analogous research avenue in the broader root-finding regime.
\end{abstract}

\noindent \textbf{Keywords:} 
 Extragradient methods; Polyak step-sizes, monotone root finding, adaptive step-sizes, stochastic approximation, line-search,
$(L_0,L_1)$-Lipschitz continuity

\vspace{.2cm}

\noindent \textbf{Mathematics Subject Classification} 65K15, 47H05, 90C33, 62L20
\newpage

\tableofcontents
\newpage

\section{Introduction}
\label{sec:introduction}

Let $F:\R^d\to\R^d$ be a single-valued operator. In this work, we consider the root-finding problem:
\begin{equation}
    \text{find }x_*\in\R^d
    \quad\text{such that}\quad
    F(x_*)=0.
    \label{eq:root_finding}
\end{equation}  
We assume throughout that the solution set
$\mathcal{X}_*\eqdef\{x\in\R^d:F(x)=0\}$ is nonempty. 
Problem~\eqref{eq:root_finding}, also known as a system of nonlinear equations when $F$ is nonlinear, encompasses several important classes of optimization and equilibrium problems, as it encodes the first-order optimality conditions for smooth convex minimization, saddle-point conditions for convex-concave min-max problems, and equilibrium conditions for smooth multiplayer games. 

The breadth of the root-finding problem has motivated the development of an extensive algorithmic literature. 
Classical approaches include the proximal point method for maximal monotone operators \citep{rockafellarMonotoneOperatorsProximal1976}, the extragradient method \citep{Korpelevich1976_extragradient}, Popov's method \citep{Popov1980_modification}, forward-backward-forward splitting \citep{tsengModifiedForwardbackwardSplitting2000}, and the non-Euclidean prox and
dual-extrapolation methods \citep{nemirovskiProxmethodRateConvergence2004,nesterovDualExtrapolationIts2007}. 
More recent developments include the optimism interpretation and reflection methods that use one new operator evaluation per iteration \citep{RakhlinSridharan2013_optimization,hsiehConvergenceSinglecallStochastic2019,gorbunovConvergenceProximalPoint2023, ChoudhuryGorbunovLoizou2023_singlecall}, 
accelerated methods based on adapting the optimized Halpern iteration \citep{Halpern1967_fixed,SabachShtern2017_first,Lieder2021_convergence,Kim2021_accelerated,contrerasOptimalErrorBounds2023} for proximal/fixed-point settings to the explicit/root-finding settings \citep{Diakonikolas2020_halpern,yoonAcceleratedAlgorithmsSmooth2021,LeeKim2021_fast,Tran-DinhLuo2021_halperntype,tran-dinhHalpernsFixedpointIterations2024} or identifying distinct acceleration mechanisms \citep{sedlmayerFastOptimisticMethod2023,yoonOptimalAccelerationMinimax2024,botFastOptimisticGradient2025,yoonHinvarianceTheoryComplete2026,yoonTheoryCompositionDuality2026}, and their stochastic or randomized-coordinate extensions for large-scale problems \citep{cai2022stochastic,tran2024variance,caiVarianceReducedHalpern2024,chenNearoptimalAlgorithmsMaking2024,tran2026randomized,yoon2026direct}. 
In this work, we are also interested in the natural stochastic setting where the objective operator $F$ is expressed as a finite sum $F(x) = \frac{1}{n} \sum_{i=1}^n F_i (x)$ and one can only access the evaluations of sample operators $F_i$~\cite{loizou2020stochastic,loizou2021stochastic,gorbunovStochasticExtragradientGeneral2022,ChoudhuryGorbunovLoizou2023_singlecall}.

Our analysis focuses on monotone operators, the standard setting for extragradient methods. We also consider strong monotonicity, under which sharper convergence guarantees can be obtained. We recall these notions below.

\begin{definition}
    \label{assume:strong_monotone}
    An operator $F$ is called \emph{monotone}, if for all $x, y \in\R^d$,
        \begin{equation}\label{eq:monotone}
           \la F(x) - F(y), x - y \ra \geq 0.
        \end{equation}
    For $\mu>0$, we say that $F$ is \emph{$\mu$-strongly monotone} if, for all $x, y \in\R^d$,
\begin{equation}\label{eq:strong_monotone}
        \la F(x) - F(y), x - y \ra \geq \mu \|x - y\|^2.
    \end{equation}
\end{definition} 

In parts of our stochastic analysis, we impose these properties samplewise; that is, each \(F_i\) is monotone or \(\mu_i\)-strongly monotone. These operator conditions arise naturally in the optimization and equilibrium models captured by the root-finding formulation, as we now illustrate.

First, consider unconstrained smooth convex minimization. 
Indeed, if $f\colon \R^d\to\R$ is continuously differentiable and ($\mu$-strongly) convex, then $\nabla f$ is a ($\mu$-strongly) operator and $x_*$ is a minimizer if and only if $\nabla f(x_*)=0$.

Another important special case is the convex-concave min-max problem
\begin{equation}
    \underset{x^1\in\R^{d_1}}{\text{minimize}} \,\,
    \underset{x^2\in\R^{d_2}}{\text{maximize}} 
    \quad
    g(x^1,x^2),
    \label{eq:min_max}
\end{equation}
where $g\colon \R^{d_1}\times\R^{d_2}\to\R$ is continuously differentiable. 
Writing $x=(x^1,x^2)\in\R^d$ with $d=d_1+d_2$, define the saddle operator
\begin{equation}
    F(x)
    \eqdef
    \begin{pmatrix}
        \nabla_{x^1}g(x^1,x^2)\\
        -\nabla_{x^2}g(x^1,x^2)
    \end{pmatrix} .
    \label{eq:saddle_operator}
\end{equation}
If $g$ is convex in $x^1$ and concave in $x^2$, then $F$ is monotone; if $g$ is additionally $\mu$-strongly convex in $x^1$ and $\mu$-strongly concave in $x^2$, then $F$ is $\mu$-strongly monotone. 
A point $x_*=(x_*^1,x_*^2)$ satisfies $F(x_*)=0$ if and only if it is a saddle point of $g$:
$g(x_*^1,x^2) \leq g(x_*^1,x_*^2) \leq g(x^1,x_*^2)$ for every $x^1\in\R^{d_1}$ and $x^2\in\R^{d_2}$. 
This operator formulation underlies much of the literature on convex-concave min-max optimization.  See, for example, \cite{nemirovskiProxmethodRateConvergence2004,nesterovDualExtrapolationIts2007,juditskySolvingVariationalInequalities2011,gorbunovStochasticExtragradientGeneral2022}. Such problems have a long history in mathematical programming and game theory and have become increasingly prominent in machine learning through applications including adversarial training \citep{MadryMakelovSchmidtTsiprasVladu2018_deep}, generative adversarial networks \citep{goodfellowGenerativeAdversarialNets2014,GidelBerardVignoudVincentLacoste-Julien2019_variational}, and distributionally robust optimization \citep{namkoongStochasticGradientMethods2016}.

The same operator viewpoint extends to smooth multiplayer games. 
Suppose each player $j\in\{1,\ldots,m\}$ controls $x^j$ to minimize a loss $\ell^j(x^j,x^{-j})$, where $x^{-j}$ denotes the decision variables of the other players. 
For $x = (x^1, \dots, x^m)$, the corresponding pseudo-gradient operator is
\begin{equation}
    F(x)
    \eqdef
    \begin{pmatrix}
        \nabla_{x^1}\ell^1(x)\\
        \vdots\\
        \nabla_{x^m}\ell^m(x)
    \end{pmatrix}.
    \label{eq:pseudogradient}
\end{equation}
In an unconstrained game in which each player's loss function $\ell^j$ is convex in their own decision variable $x^j$, the zeros of~\eqref{eq:pseudogradient} characterize Nash
equilibria. 
Games with a monotone pseudo-gradient, commonly referred to as monotone games, therefore fall within the problem formulation~\eqref{eq:root_finding} \citep{rosen1965existence,facchinei2003finite}. 
The min-max problem~\eqref{eq:min_max} is the two-player, zero-sum special case obtained by taking $\ell^1=g$ and $\ell^2=-g$. 
Multiplayer game formulations are particularly relevant in multi-agent learning, distributed decision making, and reinforcement learning \citep{zhangMultiAgentReinforcementLearning2021,sokotaUnifiedApproachReinforcement2023}. 

\paragraph{Extragradient method.} 
Although convex minimization, convex-concave min-max optimization, and monotone games all admit root-finding reformulations, their optimization dynamics are different.
Unlike in convex minimization, where $F=\nabla f$ is the gradient of a scalar potential, saddle operators or pseudo-gradients in games need not be gradients, and may generally contain rotational components.
Consequently, the direct forward iteration $x_{k+1}=x_k-\eta_k F(x_k)$ may fail to converge under monotonicity and Lipschitz continuity of $F$ alone.
A simple rotation operator $F(u,v) = (-v,u)$ on 2D is a canonical example, where the forward iteration moves away from the unique root for every positive constant step-size \citep{GidelBerardVignoudVincentLacoste-Julien2019_variational,AzizianMitliagkasLacoste-JulienGidel2020_tight}.

The Extragradient method \eqref{eq:EG} is a foundational adjustment mechanism designed for monotone inclusion problems by \citet{Korpelevich1976_extragradient}:
\begin{align}\label{eq:EG}
\begin{split}
\textstyle
    \text{Extrapolation step:} & \qquad \hx_k = x_k - \gamma_k F(x_k) \notag\\  
    \text{Update step:} & \quad x_{k+1} = x_k - \alpha_k F(\hx_k) .
\end{split}
\tag{\algname{EG}}
\end{align}
Here $\gamma_k, \alpha_k >0 $ are respectively called the \emph{extrapolation} and \emph{update} step-sizes. 
By using the operator evaluated at the extrapolated point $\hx_k$ as a ``descent direction'', \algname{EG} overcomes the outward divergence due to rotation.
This idea has had a significant impact on the min-max optimization and monotone operator literature.
Popov, optimistic-gradient, and forward-reflected methods \citep{Popov1980_modification,RakhlinSridharan2013_optimization,hsiehConvergenceSinglecallStochastic2019,gorbunovConvergenceProximalPoint2023,ChoudhuryGorbunovLoizou2023_singlecall} were developed and analyzed as a single-call variants of \algname{EG} reusing previous operator evaluations rather than evaluating a fresh new operator value at an extrapolated point. 
Mirror-Prox and dual extrapolation algorithms extend \algname{EG}-type update rules to non-Euclidean
geometries \citep{nemirovskiProxmethodRateConvergence2004,nesterovDualExtrapolationIts2007}.
Interestingly, \algname{EG} can be interpreted \cite{MokhtariOzdaglarPattathil2020_unified,MokhtariOzdaglarPattathil2020_convergence} as an approximation of the proximal point method \citep{rockafellarMonotoneOperatorsProximal1976}.
Recent accelerated algorithms combined anchoring with \algname{EG} or its single-call variants \cite{yoonAcceleratedAlgorithmsSmooth2021,LeeKim2021_fast,Tran-DinhLuo2021_halperntype,CaiZheng2023_accelerated}, and analogous interpretation of those algorithms as approximate Halpern iteration has also been established \citep{yoonAcceleratedMinimaxAlgorithms2025}.
Stochastic Extragradient (\algname{SEG}), which replaces $F(x_k)$ and $F(\hx_k)$ in the update rule of \algname{EG} with the corresponding stochastic sample operator evaluations, has also been extensively stuided \citep{juditskySolvingVariationalInequalities2011,yousefianStochasticMirrorproxAlgorithms2018,kannanOptimalStochasticExtragradient2019,GidelBerardVignoudVincentLacoste-Julien2019_variational,ChavdarovaGidelFleuretLacoste-Julien2019_reducing,mertikopoulosOptimisticMirrorDescent2019,HsiehIutzelerMalickMertikopoulos2020_explore,MishchenkoKovalevShulginRichtarikMalitsky2020_revisiting,yoon2026same}.

\paragraph{Operator regularity and classical guarantees.}
The convergence analysis and step-size selections for \algname{EG} and \algname{SEG} are commonly stated under appropriate regularity conditions on $F$ or the sample operators $F_i$.
A standard condition is the following (global) Lipschitz continuity.

\begin{definition}
    \label{def:Lipschitz}
    An operator $F$ is called $L$-Lipschitz if $\textstyle \left\|F(x) - F(y) \right\| \leq L \left\| x - y \right\|.$
\end{definition}

This is a typical assumption in the analysis of \algname{EG} and its variants
\citep{Korpelevich1976_extragradient,nemirovskiProxmethodRateConvergence2004,MokhtariOzdaglarPattathil2020_unified,GidelBerardVignoudVincentLacoste-Julien2019_variational,gorbunovLastiterateConvergenceOptimistic2022}. 
For monotone Lipschitz root-finding problems, \algname{EG} attains an $\mathcal{O}(1/K)$ ergodic gap rate, which is optimal in the general first-order oracle model up to constants
\citep{nemirovskiProxmethodRateConvergence2004,nesterovDualExtrapolationIts2007}. 
More recently, last-iterate convergence guarantees have also been established for extragradient under monotonicity and Lipschitz continuity \citep{GorbunovLoizouGidel2022_extragradient,caiFiniteTimeLastIterateConvergence2022}. Under strong monotonicity, suitable extragradient variants converge linearly with optimal condition-number dependence \citep{tseng1995linear,MokhtariOzdaglarPattathil2020_unified,AzizianMitliagkasLacoste-JulienGidel2020_tight,GorbunovLoizouGidel2022_extragradient}. 
These guarantees, together with its explicit first-order structure and direct stochastic extension, make \algname{EG} a natural algorithmic foundation for the problems considered here.

Global Lipschitz continuity can nevertheless be restrictive—the local variation of the operator may not be uniformly bounded. 
To cover a broader range of problems, we also consider the following generalized regularity classes.

\begin{definition}[H\"older continuity]
\label{def:Holder-continuity}
We say that $F$ is $(L,\holderexponent)$-H\"older (continuous) for some $0 < \holderexponent \le 1$ if, for all $x,y$,
\[
    \norm{F(x)-F(y)} \le L\norm{x-y}^\holderexponent .
\]
\end{definition}

H\"older continuity is a natural generalization of Lipschitz continuity (\cref{def:Lipschitz}), which is recovered by setting $\holderexponent=1$ in \cref{def:Holder-continuity}.
H\"older-continuous monotone variational inequalities have been studied using Mirror-Prox and \algname{EG} methods~\citep{dangConvergencePropertiesNonEuclidean2015}, including universal variants that do not require prior knowledge of the H\"older exponent or constant~\citep{stonyakinGeneralizedMirrorProx2022,klimza2024universal}. 
Recent work has also developed splitting methods for H\"older-continuous monotone inclusions~\citep{zhang2026convergence}.

\begin{definition}[$(L_0,L_1)$-Lipschitzness]
\label{def:L0L1-Lipschitz}
We say that $F$ is $(L_0,L_1)$-Lipschitz if, for all $x,y$,
\[
    \norm{F(x)-F(y)}
    \le
    \left(
        L_0+
        L_1\max_{\theta\in[0,1]}
        \norm{F(\theta x+(1-\theta)y)}
    \right)
    \norm{x-y}.
\]
\end{definition}

The $(L_0,L_1)$-Lipschitz condition recovers global $L$-Lipschitz continuity by taking $L_0=L$ and $L_1=0$, while also allowing the local variation of the operator to grow with the operator's magnitude.
Such problem classes have recently been proposed in the context of minimization \citep{zhangWhyGradientClipping2020} and quickly gained popularity; then they have been extended to variational inequalities, and root-finding problems \citep{vankovGeneralizedSmoothVariational2024,
choudhury2025extragradient}. 
In fact, \cref{def:L0L1-Lipschitz} implies the following bound \cite[Proposition~3.1]{choudhury2025extragradient} for any $x,y$, which is more useful for convergence analyses:
\begin{align}
\label{eq:L0L1-alpha1-characterization}
    \norm{F(x) - F(y)} \le \left( L_0 + L_1 \norm{F(x)} \right) \exp \left( L_1 \norm{x - y} \right) \norm{x - y} .
\end{align}

Note that monotonicity and regularity play different roles in our analysis.
Monotonicity is the principal structural condition that enables the derivation of descent property, whereas conditions like Lipschitzness, H\"older continuity or $(L_0,L_1)$-Lipschitzness quantitatively determine the range of stable step-sizes.

\paragraph{Adaptive step-sizes and Polyak's principle.}
Typical analyses of \algname{EG} and \algname{SEG} select $\gamma_k$ and $\alpha_k$ using the global Lipschitz constant of $F$ \citep{nemirovskiProxmethodRateConvergence2004,juditskySolvingVariationalInequalities2011,
GidelBerardVignoudVincentLacoste-Julien2019_variational,MishchenkoKovalevShulginRichtarikMalitsky2020_revisiting,diakonikolasEfficientMethodsStructured2021,GorbunovLoizouGidel2022_extragradient,gorbunovStochasticExtragradientGeneral2022}. 
Similarly, known step-size rules for \algname{EG} for H\"older-continuous or $(L_0,L_1)$-Lipschitz operators depend on the problem constants \citep{dangConvergencePropertiesNonEuclidean2015,vankovGeneralizedSmoothVariational2024,choudhury2025extragradient}. 
However, these parameters may be unknown, difficult to estimate, or overly conservative compared to the local geometry along the algorithm trajectory in practice.
This motivated the adaptive extragradient schemes, including the AdaGrad-type step-sizes~\citep{antonakopoulosAdaptiveExtragradientMethods2021, antonakopoulosSiftingNoiseUniversal2021}.
More classically, projection-based corrections to the update step of \algname{EG} were considered in \cite{solodovModifiedProjectiontypeMethods1996,sunClassIterativeMethods1996,iusemVariantKorpelevichsMethod1997,solodovNewProjectionMethod1999} and revisited in \cite{pethickEscapingLimitCycles2022}.
The goal of this work is aligned with these prior work: \textit{to design efficient variants of \algname{EG} and \algname{SEG} that adapts to the operator geometry without the knowledge of problem-defining parameters.}

To this end, we draw inspiration from the observation that the projection-type update for \algname{EG} studied in \cite{sunClassIterativeMethods1996,iusemVariantKorpelevichsMethod1997,solodovNewProjectionMethod1999,pethickEscapingLimitCycles2022} admits the interpretation as Polyak step-size, which was originally developed for gradient and subgradient methods in convex minimization.
More precisely, Polyak step-size in minimization is derived by minimizing a one-step upper bound on the distance to an optimum, and we show that the update step-size $\alpha_k$ in \algname{EG} can be selected with the same principle. 
This eventually agrees with $\alpha_k$ chosen via the projection interpretaion in prior work.
However, we can leverage the Polyak viewpoint to further design and analyze Polyak-type \textit{stochastic}  \algname{EG} methods, following the recent development of stochastic Polyak step-sizes from stochastic minimization literature \citep{loizouStochasticPolyakStepsize2021,orvietoDynamicsSGDStochastic2022,dorazioStochasticMirrorDescent2023,gower2021stochastic,GowerSebbouhLoizou2021_sgd, gower2025analysis,oikonomouStochasticPolyakStepsizes2025,oikonomou2025safeguarded,oikonomou2026adaptive,oikonomou2025taking}.
Table~\ref{tab:polyak_step_stochastic} summarizes the analogy between Polyak-type step-sizes in minimization and root-finding problems, and their corresponding stochastic and decreasing stochastic counterparts. 

Now we summarize our main technical contributions as follows.

\begin{table*}[!t]
        \centering
        \caption{\small
        Correspondence between Polyak-type step-sizes for gradient descent in minimization and the extragradient algorithm in root finding.
        }
        \label{tab:polyak_step_stochastic}
        \vspace{-2mm}
        \begin{threeparttable}
        \resizebox{1\textwidth}{!}{%
            \begin{tabular}{|c|c|c|}
            \hline
            Setup & \begin{tabular}{c}
                 Minimization: $\min_x \, f(x)$\\
                Gradient descent
            \end{tabular} & \begin{tabular}{c}
                 Root-finding: $\text{find}_x \, F(x) = 0$\\
                Extragradient \\
            \end{tabular} 
            \\
            \hline\hline
            \rule{0pt}{15pt} Deterministic & \begin{tabular}{c}
                $\eta_k = \frac{f(x_k) - \min_x f(x)}{ \left\| \nabla f\left(x_k \right) \right\|^2}$ \\ Polyak step-size~\citep{polyak1987introduction}
            \end{tabular} & \cellcolor{bgcolor2} \begin{tabular}{c}
                 $\alpha_k = \frac{\la F(\hx_k), x_k - \hx_k \ra}{\left\|F(\hx_k) \right\|^2}$ \\ \algname{PolyakEG}: \cref{theorem:deterministic-master-theorem}
            \end{tabular}  \\[5pt]
            \hline
            \rule{0pt}{20pt} \begin{tabular}{c} Stochastic \\ Interpolated \end{tabular} & \begin{tabular}{c}
                $\eta_k = \frac{f_{\mathcal{S}_k}(x_k) - \min_x f_{\mathcal{S}_k}(x)}{ \left\| \nabla f_{\mathcal{S}_k} \left(x_k \right) \right\|^2}$ \\ Stochastic Polyak step-size (\algname{SPS})~\citep{loizouStochasticPolyakStepsize2021}
            \end{tabular} & \cellcolor{bgcolor2} \begin{tabular}{c}
                 $\alpha_k = \frac{\la F_{\mathcal{S}_k}(\hx_k), x_k - \hx_k \ra}{\left\|F_{\mathcal{S}_k}(\hx_k) \right\|^2}$ \\ \algname{PolyakSEG}: \cref{theorem:PolyakSEG-convergence1}
            \end{tabular}  \\[10pt]
            \hline
            \rule{0pt}{20pt} Stochastic & \begin{tabular}{c}
                $\eta_k = \frac{1}{c_{k+1}} \min \left\{ \frac{f_{\mathcal{S}_k}(x_k) - \min_x f_{\mathcal{S}_k}(x)}{c \left\| \nabla f_{\mathcal{S}_k}(x_k) \right\|^2}, c_k \eta_{k-1}\right\}$ \\ Decreasing SPS (\algname{DecSPS})~\citep{orvietoDynamicsSGDStochastic2022}
            \end{tabular} &  \cellcolor{bgcolor2} \begin{tabular}{c} $\alpha_k = \min \left\{ \frac{\la F_{\mathcal{S}_k}(\hx_k), x_k - \hx_k \ra}{ \left\|F_{\mathcal{S}_k}(\hx_k) \right\|^2}, \alpha_{k-1}\right\}$ \\ \algname{DecPolyakSEG}: \cref{theorem:DecPolyakSEGLS_monotone} \end{tabular} \\[10pt]
            \hline
        \end{tabular}%
        }
        \end{threeparttable}
\end{table*}

\subsection{Main Contributions}

\begin{table}[t]
    \centering
    \small
    \begin{threeparttable}
        \caption{\small
            Comparison with representative prior work on deterministic extragradient-type algorithms. 
            The ``H\"older rates'' and ``$(L_0, L_1)$ rates'' columns indicate whether an explicit nonasymptotic analyses under the respective regularity conditions (beyond global Lipschitz continuity) were provided.
        }
        \label{tab:deterministic_regularity_comparison}

        \setlength{\tabcolsep}{6pt}
        \renewcommand{\arraystretch}{1.25}
        \renewcommand{\tabularxcolumn}[1]{m{#1}}

        \begin{tabularx}{\linewidth}{@{}
            >{\raggedright\arraybackslash}X
            *{3}{>{\centering\arraybackslash}m{0.13\linewidth}}
        @{}}
            \toprule
            Representative works
            & \shortstack{H\"older\\rates}
            & \shortstack{$(L_0,L_1)$\\rates}
            & \shortstack{Polyak-type\\correction} \\
            \midrule

            \textbf{Projection-based step-sizes for EG}\newline
            \citet{sunClassIterativeMethods1996};
            \citet{iusemVariantKorpelevichsMethod1997};
            \citet{solodovNewProjectionMethod1999}\tnote{a};
            \citet{pethickEscapingLimitCycles2022}
            & \xmark & \xmark & \cmark \\
            \addlinespace[4pt]

            \textbf{H\"older-continuous EG and Mirror-Prox}\newline
            \citet{dangConvergencePropertiesNonEuclidean2015};
            \citet{stonyakinGeneralizedMirrorProx2022};
            \citet{klimza2024universal}
            & \cmark & \xmark & \xmark \\
            \addlinespace[4pt]

            \textbf{EG for $(L_0, L_1)$-Lipschitz problems}\newline
            \citet{vankovGeneralizedSmoothVariational2024}\tnote{b};
            \citet{choudhury2025extragradient}
            & \xmark & \cmark & \xmark \\
            \midrule

            \rowcolor{bgcolor2}
            \textbf{\algname{PolyakEG} / \algname{PolyakEG-LS} (this work)}
            & \cmark & \cmark & \cmark \\
            \bottomrule
        \end{tabularx}

        \begin{tablenotes}[flushleft]
            \footnotesize

            \item[a] \cite{sunClassIterativeMethods1996, iusemVariantKorpelevichsMethod1997, solodovNewProjectionMethod1999} prove convergence under (generalized) monotonicity and continuity, but without explicit rates for H\"older-continuous, $(L_0,L_1)$-Lipschitz or stochastic problems. 

            \item[b] \citet{vankovGeneralizedSmoothVariational2024} uses a condition called $p$-quasi-sharpness in their convergence analysis.
        \end{tablenotes}
    \end{threeparttable}
\end{table}

\begin{itemize}[leftmargin=*] \setlength{\itemsep}{2pt} 
\item \textbf{\algname{PolyakEG}: Polyak-inspired Extragradient Update Step-size.} 
In Section~\ref{sec:motivation}, we derive the update step-size $ \alpha_k = \frac{ \la F(\hx_k),x_k-\hx_k\ra }{ \|F(\hx_k)\|^2 }$ via Polyak-type derivation, by minimizing an upper bound on $\|x_{k+1}-x_*\|^2$, in direct analogy with the classical Polyak step-size for gradient methods.
We then show that the projection-based correction for \algname{EG} introduced by \cite{sunClassIterativeMethods1996,iusemVariantKorpelevichsMethod1997,solodovNewProjectionMethod1999,pethickEscapingLimitCycles2022} yields the same formula. 
This interpretation serves as a basis for the design and analysis principle for the stochastic extensions developed later in the paper.

\item \textbf{Unified deterministic analysis.}
    We provide a unified convergence result (\cref{theorem:deterministic-master-theorem}) for \algname{PolyakEG} (\cref{alg:PolyakEG}) based on a critical condition (\cref{def:critical-condition-deterministic}) controlling the variation of the operator along the extrapolation step. 
    The analysis does not require global Lipschitz continuity and simultaneously captures distinct regularity conditions, and provides sublinear residual convergence under monotonicity and linear convergence under strong monotonicity.

    \begin{itemize}[leftmargin=*]
        \setlength{\itemsep}{1pt}

        \item
        \textbf{$L$-Lipschitz, $(L,\holderexponent)$-H\"older and $(L_0,L_1)$-Lipschitz operators.}
        We show that the critical condition accommodates constant extrapolation steps for globally Lipschitz operators and adaptive extrapolation steps for H\"older-continuous or
        $(L_0,L_1)$-Lipschitz operators. 
        Consequently, we newly provide explicit quantitative guarantees for the last two problem classes (see \cref{tab:deterministic_regularity_comparison}).

        \item
        \textbf{Line-search implementations.}
        We develop a unified line-search variant called \algname{PolyakEG-LS} (\cref{alg:PolyakEG_linesearch_unified}) encompassing Lipschitz, H\"older-continuous and $(L_0,L_1)$-Lipschitz operators. 
        This algorithm directly enforces the critical condition via backtracking, and does not require prior knowledge of $L, \holderexponent, L_0$ or $L_1$. 
        In all settings, we bound the total complexity of backtracking trials, which shows that \algname{PolyakEG-LS} retains the complexity comparable to choosing $\gamma_k$ explicitly based on the problem-defining constants, while being parameter-free (\cref{theorem:L0L1-linesearch-deterministic}).
    \end{itemize}

 \item
    \textbf{Stochastic extensions.}
    We extend \algname{PolyakEG} to the stochastic root-finding problems with finite-sum structure $F(x)=\frac{1}{n} \sum_{i=1}^n F_i(x)$, where we draw a mini-batch $\cS_k \subseteq [n]$ with replacement at each iteration $k$, and use the mini-batch operator $F_{\cS_k} = \frac{1}{|\cS_k|} \sum_{i \in \cS_k} F_i$ for making the updates. 
    In particular, we propose and analyze two algorithms: \algname{PolyakSEG} (Algorithm~\ref{alg:PolyakSEG}) and \algname{DecPolyakSEG} (Algorithm~\ref{alg:decPolyakSEG}). 
    
    \begin{itemize}[leftmargin=*]
        \setlength{\itemsep}{1pt}

        \item When all sample operators are monotone and share a common solution, we show that         \algname{PolyakSEG} achieves a $\mathcal O(1/K)$ sublinear convergence on squared residual norm. 
        When the expected strong-monotonicity modulus is positive, we additionally establish linear convergence (\cref{theorem:PolyakSEG-convergence1}). 
        We further show by an example that, with nonvanishing step-sizes, the interpolation (existence of a common solution) assumption cannot in general be removed (\cref{prop:PolyakSEG-noninterpolation}).

        \item
        For the general stochastic regime, we introduce \hyperref[alg:decPolyakSEG]{\algname{DecPolyakSEG}}, which combines decreasing extrapolation and update step-sizes with Polyak-type update. 
        When $F$ is monotone and each $F_i$ is Lipschitz, we establish an $\mathcal O(K^{-1/2})$ sublinear convergence on the expected squared residual under a trajectory localization condition (\cref{corollary:DecPolyakSEG-explicit-residual-rate}). 
        We show that this rate is attained both by a predetermined decreasing-step schedule and by a line-search variation. 
        We show that the localization condition can be guaranteed with additional problem structures such as strong monotonicity of the sample operators (\cref{proposition:DecPolyakSEG_strong_monotone}).
    \end{itemize}

 \item
    \textbf{Numerical experiments.}
    In Section~\ref{sec:numerical_experiment}, we evaluate the proposed Polyak-type algorithms on several representative root-finding and min-max problems. 
    Our experiments examine the benefit of Polyak-type correction in \algname{EG}, uniformly competitive performance over globally Lipschitz, H\"older-continuous and $(L_0,L_1)$-Lipschitz settings, and the promising empirical performance of on stochastic problems.
\end{itemize}

\section{Polyak-Type Update for Extragradient: Derivation and Geometry}
\label{sec:motivation}

In this section, we derive the \algname{PolyakEG} update step-size, following the viewpoint of \citet{polyak1987introduction}.
This is done by optimizing the progression toward a solution within the standard one-step upper bound in \algname{EG}.
Then we clarify its connection to the projection interpretation onto separating halfspace introduced in prior work \cite{sunClassIterativeMethods1996, iusemVariantKorpelevichsMethod1997, solodovNewProjectionMethod1999}.

\subsection{Derivation of the Polyak Update Step-size}

We first revisit how the classical Polyak step-size is derived for Gradient Descent, and then take the analogous approach to provide for \algname{EG}. 

\paragraph{Polyak step-size for Gradient Descent.} Gradient Descent (\algname{GD}) is arguably the most widely recognized method for minimization problems: $\min_{x \in \R^d} f(x)$.
\algname{GD} updates iterates by: $x_{k+1} = x_k - \eta_k \nabla f(x_k)$, where $\eta_k$ is the step-size. 
Classical convergence analysis for convex, $L$-smooth function $f(x)$ (whose gradient $\nabla f$ is $L$-Lipschitz) assumes a constant step-size $\eta_k \le \nicefrac{1}{L}$.

On the other hand, \citet{polyak1987introduction} proposed an adaptive strategy of choosing $\eta_k$ that minimizes the upper bound of $\|x_{k+1}-x_*\|^2$ for $x_* \in \argmin_{x\in\reals^d} f(x)$:
$$ \|x_{k+1} - x_*\|^2 \leq \|x_k - x_*\|^2 - 2 \eta_k \left( f(x_k) - f(x_*) \right) + \eta_k^2 \| \nabla f(x_k)\|^2 $$ 
which follows from the update rule of \algname{GD} and convexity of $f$. 
Unless $x_k$ is already a minimizer of $f$ so that $\nabla f(x_k) = 0$, the right-hand side is minimized with \begin{equation}
\label{Polyakminimization}
    \eta_k = \frac{f(x_k) - f(x_*)}{\| \nabla f(x_k)\|^2},
\end{equation}
which is called the \textit{Polyak step-size} for \algname{GD}. 
The step-size \eqref{Polyakminimization} is computable \textit{only when the optimal function value $f(x_*)$ is accessible}. 
While this seems to be a strong requirement, there are several classes of important problems where $f(x_*)=0$ is known even if $x_*$ is unknown, such as consistent linear systems, convex feasibility problems or overparametrized learning problems~\citep{polyak1987introduction, hazan2019revisiting, loizouStochasticPolyakStepsize2021}.

\paragraph{\textbf{Polyak step-size for Extragradient.}}
We can follow the analogous strategy of optimally selecting the update step-size $\alpha_k$ for \hyperref[eq:EG]{\algname{EG}}.
Fix $x_*\in\mathcal X_*$, and recall $\hx_k=x_k-\gamma_kF(x_k)$. 
For any $\alpha_k \geq 0$, the update $x_{k+1}=x_k-\alpha_kF(\hx_k)$ satisfies:
\begin{align*}
    \sqnorm{x_{k+1} - x_*} & = \sqnorm{x_k - x_* } - 2\alpha_k \inprod{F(\hx_k)}{x_k - \hx_k} - 2\alpha_k \inprod{F(\hx_k)}{\hx_k - x_*} + \alpha_k^2 \sqnorm{F(\hx_k)} \\
    & \le \sqnorm{x_k - x_* } - 2\alpha_k \inprod{F(\hx_k)}{x_k - \hx_k} + \alpha_k^2 \sqnorm{F(\hx_k)}
\end{align*}
where the second lines uses monotonicity of $F$.
Now, following similar approach to \citet{polyak1987introduction}, we minimize the bound on the right hand side with respect to $\alpha_k \ge 0$, which is attained with $\alpha_k = \frac{[\la F(\hx_k), x_k - \hx_k \ra]_+}{\| F(\hx_k)\|^2}$, where $[t]_+\eqdef\max\{t,0\}$. 
Later in Section in Section~\ref{sec:conv_ana1}, we show that with appropriate choice of $\gamma_k$, we have $\la F(\hx_k), x_k - \hx_k \ra>0$ unless $\hx_k$ is a solution where $F(\hx_k) = 0$. 
Then we can remove the operation $[\cdot]_+$, and the update step-size simply becomes
\begin{equation}
\label{eq:polyakEG-step-size}
    \boxed{\alpha_k = \frac{\la F(\hx_k), x_k - \hx_k \ra}{\| F(\hx_k)\|^2}.}
\end{equation}
We highlight that this derivation inherits the spirit of Polyak step-size for \algname{GD}, but \textit{does not require the knowledge of optimal function values or other inaccessible quantities}. 

We denote \algname{EG} with update step-size~\eqref{eq:polyakEG-step-size} by \hyperref[alg:PolyakEG]{\algname{PolyakEG}} (Algorithm \ref{alg:PolyakEG}). 
We highlight once again that this step-size selection is not new: early works including \cite{sunClassIterativeMethods1996, iusemVariantKorpelevichsMethod1997, solodovNewProjectionMethod1999} considered a version of it, and some recent works also studied it~\citep{giselssonNonlinearForwardBackwardSplitting2021,pethickEscapingLimitCycles2022}. 
However, to our knowledge, the interpretation of~\eqref{eq:polyakEG-step-size} as a Polyak-type step-size and the connection between the two literature has not been established before. 
In our work, we leverage this new insight to extend these ideas to stochastic settings and design new algorithms.

\begin{algorithm}[H]
\small
    \caption{\algname{\textcolor{PineGreen}{PolyakEG}}}
    \label{alg:PolyakEG}
    \begin{algorithmic}[1]
        \REQUIRE Initial point $x_0 \in \R^d$ and a rule for choosing extrapolation steps $\{\gamma_k\}_{k=0}^\infty$.
        \FOR{$k = 0, 1,...,K$}
        \STATE $\hx_k = x_k - \gamma_k F(x_k)$.
        \vspace{2mm}
        \STATE \colorbox{green!20}{$\alpha_k = \frac{\la F(\hx_k), x_k - \hx_k \ra}{\|F(\hx_k)\|^2}$.}
        \vspace{2mm}
        \STATE $x_{k+1} = x_k - \alpha_k F(\hx_k)$.
        \ENDFOR
    \end{algorithmic}
\end{algorithm}

\begin{tikzfigure}[H]
    \resizebox{0.6\linewidth}{!}{%
    \begin{tikzpicture}[
        x=1cm, y=1cm,
        font=\normalsize,
        line cap=round,
        line join=round,
        >={Stealth[length=2.1mm,width=1.4mm]}
    ]
        \definecolor{pegGreen}{RGB}{48,132,55}
        \definecolor{pegShade}{RGB}{235,243,236}
        \definecolor{pegRed}{RGB}{225,49,48}
        \definecolor{pegBlue}{RGB}{0,92,190}

        \coordinate (Htop)    at (3.70, 3.95);
        \coordinate (Hbottom) at (7.20,-2.70);
        \coordinate (xh)      at (6.70,-1.75);
        \coordinate (xk)      at (0,0);
        \coordinate (xs)      at (6.85,0.87);

        \coordinate (xnext) at ($(Htop)!(xk)!(Hbottom)$);

        \coordinate (normalTip)
            at ($(xh)+0.26*(xk)-0.26*(xnext)$);

        \fill[pegShade]
            (Htop) -- (8.75,3.95) -- (8.75,-2.70)
            -- (Hbottom) -- cycle;

        \draw[
            pegGreen,
            line width=0.75pt,
            dash pattern=on 3pt off 3pt
        ] (Htop) -- (Hbottom);

        \begin{scope}[shift={(xs)}]
            \path[
                draw=pegBlue,
                fill=pegBlue!12!pegShade,
                line width=0.7pt
            ]
                (-0.82,-0.18)
                .. controls (-0.82, 0.28) and (-0.34, 0.63) .. (0.21, 0.63)
                .. controls ( 0.79, 0.63) and ( 1.30, 0.36) .. (1.30,-0.09)
                .. controls ( 1.30,-0.55) and ( 0.92,-0.89) .. (0.34,-0.89)
                .. controls (-0.26,-0.89) and (-0.82,-0.64) .. cycle;
        
            \node[pegBlue] at (0.60,-0.50) {$\cX_*$};
        \end{scope}

        \draw[gray,line width=0.85pt,->]
            (xk) -- (xh);

        \node at (2.80,-1.20)
            {$-\gamma_k F(x_k)$};

        \draw[pegRed,line width=0.8pt,->]
            (xk) -- (xnext);

        \node[pegRed,anchor=south] at (1.75,1.55)
            {$-\dfrac{\langle F(\hx_k),x_k-\hx_k\rangle}
                      {\|F(\hx_k)\|^2}\,F(\hx_k)$};

        \draw[pegRed,line width=0.8pt,->]
            (xh) -- (normalTip)
            node[midway,below=4pt] {$F(\hx_k)$};

        \node[pegGreen,anchor=west] at (4.94,1.87)
            {$\langle F(\hx_k),x-\hx_k\rangle\leq 0$};

        \foreach \p in {xk,xh,xnext,xs}
            \filldraw[
                draw=black,
                fill=pegBlue,
                line width=0.45pt
            ] (\p) circle[radius=1.25pt];

        \node[pegBlue,below=3pt]
            at (xk) {$x_k$};

        \node[pegBlue,anchor=west]
            at ($(xh)+(0.10,-0.10)$) {$\hx_k$};

        \node[pegBlue,above right=2pt]
            at ($(xnext)+(-0.10,-0.10)$) {$x_{k+1}$};

        \node[pegBlue,right=3pt]
            at (xs) {$x_*$};

    \end{tikzpicture}%
    }

    \caption{\small
        Geometric interpretation of \algname{PolyakEG}.
        The extrapolated point $\hx_k=x_k-\gamma_k F(x_k)$ defines a halfspace $\mathcal D(\hx_k)$ (colored in green), containing the solution set $\cX_*$ (colored in blue).
        The vector $F(\hx_k)$ is an outward normal to $\mathcal D(\hx_k)$.
        \algname{PolyakEG} projects the current iterate $x_k$ onto $\mathcal{D}(\hx_k)$, i.e., $x_{k+1}=\Pi_{\mathcal D(\hx_k)}(x_k)$.
    }
    \label{fig:polyakEG_geometric1}
\end{tikzfigure}

\subsection{Geometric Interpretation of Polyak Step-size for Extragradient}

The preceding derivation characterizes \algname{PolyakEG} through a one-step distance minimization. 
\Cref{fig:polyakEG_geometric1} provides an alternative geometric illustration of \algname{PolyakEG} due to \citet{solodovNewProjectionMethod1999}.
Given the extrapolation point $\hx_k$, assuming $F(\hx_k) \neq 0$, consider the 
\begin{eqnarray*}
    \mathcal{D} \left(\hx_k \right) \eqdef \left\{ x \in \R^d \mid \la F(\hx_k), \hx_k - x \ra \geq 0 \right\} .
\end{eqnarray*}
This is a halfspace, depicted as the green-shaded region in \cref{fig:polyakEG_geometric1}, which contains any solution $x_* \in \cX_*$, by monotonicity of $F$.
Then, clearly, $x_{k+1} = x_k - \frac{\la F(\hx_k), x_k - \hx_k \ra}{\|F(\hx_k)\|^2} F(\hx_k)$ of \algname{PolyakEG} is a projection onto this halfspace, provided that $\la F(\hx_k), x_k - \hx_k \ra \ge 0$.
This aligns with the intuition that $x_{k+1} = x_k - \alpha_k F(\hx_k)$ for $\alpha_k$ that minimizes an upper bound on $\sqnorm{x_{k+1} - x_*}$; near $\hx_k$, by only using the local observation $F(\hx_k)$ and the fact that $F$ is monotone, the halfspace $\mathcal{D} \left(\hx_k \right)$ is the best geometric estimate of the region containing $\cX_*$. 
As \algname{EG} starts from $x_k$ and updates it in the direction of $-F(\hx_k)$, and selecting the point where the update direction meets the boundary of $\mathcal{D} \left(\hx_k \right)$ is the geometrically tightest choice.

\section{Deterministic Setting}
\label{sec:conv_ana1}

In this section, we provide an analysis of \hyperref[alg:PolyakEG]{\algname{PolyakEG}} in the deterministic case. 
We provide a unified theorem under a critical condition, which captures a general requirement on the extrapolation step-size $\gamma_k$ for which \algname{PolyakEG} update step-size $\alpha_k$ is effective.
As a result, our main convergence theorem provides guarantees for \algname{PolyakEG} without assuming  global Lipschitz continuity of $F$; in particular, it covers the broader classes of $(L,\holderexponent)$-H\"older or $(L_0, L_1)$-Lipschitz operators under the same unified framework.
Depending on the problem class, $\gamma_k$ can be constant, adaptive, or determined via line-search so that the critical condition is satisfied.

\subsection{Critical Condition for $\gamma_k$}
While \hyperref[alg:PolyakEG]{\algname{PolyakEG}} (\Cref{alg:PolyakEG}) specifies the update step-size $\alpha_k$, it does not specify the extrapolation step-size $\gamma_k$.  
In our approach, we propose a critical condition that captures the behavior of several choices for $\gamma_k$. 

\begin{definition} 
\label{def:critical-condition-deterministic}
We say $\gamma_k > 0$ satisfies the critical condition for some $A \in (0, 1]$ if
\begin{equation}\label{eqn:gamma-critical-condition}
    \norm{F(\hat{x}_k) - F(x_k)} \le A \norm{F(x_k)}
\end{equation}
for $ \hx_k = x_k - \gamma_k F(x_k)$.
\end{definition}
The importance of the critical condition~\eqref{eqn:gamma-critical-condition} for selecting the extrapolation step-size $\gamma_k$ is that it provides a natural lower bound on the update step-size $\alpha_k$, thereby guaranteeing sufficient progress per iteration. 
We defer the proof of \cref{lemma:alpha-lower-bound} to \cref{sec:stochastic}, where we prove its generalized version, \cref{lemma:PolyakSEG-LS-key-properties}.

\begin{lemma}
\label{lemma:alpha-lower-bound}
If the critical condition \eqref{eqn:gamma-critical-condition} is satisfied with $A \in (0,1]$, then we have $$\textstyle \inprod{F(\hat{x}_k)}{F(x_k)} \ge \frac{1}{2} \left( \sqnorm{F(\hat{x}_k)} + (1 - A^2) \sqnorm{F(x_k)} \right)$$
which implies $\alpha_k \ge \frac{\gamma_k}{1+A}$.
\end{lemma}

Note that the critical condition \eqref{eqn:gamma-critical-condition} relies solely on the local geometry of $F$ near $x_k$.
For example, under $L$-Lipschitzness (Definition~\ref{def:Lipschitz}), $0 < \gamma_k \le \frac{A}{L}$ will satisfy \eqref{eqn:gamma-critical-condition}, but a larger $\gamma_k$ might be admissible if the local Lipschitz constant of $F$ is smaller near $x_k$.
This illustrates the flexibility of \eqref{eqn:gamma-critical-condition} to include different choices of $\gamma_k$. 

\subsubsection{\algname{PolyakEG-LS}: Parameter-free variant with backtracking line-search}

\cref{prop:gamma-critical-condition-sufficient} shows several cases where $\gamma_k$ can be chosen based on the parameters determining the problem class and $A$. 
On the other hand, we can simply enforce the critical condition to hold via taking it as an escaping criterion in backtracking line-search, as in \cref{alg:PolyakEG_linesearch_unified} (\algname{PolyakEG-LS}).
\algname{PolyakEG-LS} is parameter-free in the sense that it works for any one of $L$-Lipschitz, $(L,\holderexponent)$-H\"older or $(L_0,L_1)$-Lipschitz settings without requiring the knowledge of parameters such as $L, \holderexponent, L_0$ or $L_1$.
Its design is inspired by \cite{vyguzov2026frankwolfealgorithmsl0l1smooth} for $(L_0, L_1)$-smooth minimization, using the similar idea of increasing the estimate $\lambda_k^j$ for $L_j$ (for $j=0,1$) more aggressively when it contributes less to the denominator $\lambda_k^0 + \lambda_k^1 \norm{F(x_k)}$, while taking turns to adjust each estimate via the `toggle' variable.
While we only allow $\lambda_k^j$ to increase and $\gamma_k$ to consequently decrease for simplicity of exposition, one can easily modify the line-search incorporate decreasing $\lambda_k^j$ as in \cite{vyguzov2026frankwolfealgorithmsl0l1smooth} or \cite{nesterovGradientMethodsMinimizing2013}.

When the problem is $L$-Lipschitz or $(L,\holderexponent)$-H\"older, we can simply set $\lambda_{-1}^1 = 0$ and all $\textrm{toggle}=1$ cases can be safely ignored.
In that case, the line-search loop will simply shrink the step-size by the factor of $\beta - 1$ until the critical condition is satisfied.

\begin{algorithm}[t]
\small
    \caption{\algname{\textcolor{PineGreen}{PolyakEG-LS}}}
    \label{alg:PolyakEG_linesearch_unified}
    \begin{algorithmic}[1]
        \REQUIRE Initial point $x_0 \in \R^d$, initial line-search estimates $\lambda_{-1}^0 > 0, \lambda_{-1}^1 \ge 0$, line-search factor $\beta > 2$, $0 < A \le 1$ and $\nu_A > 0$ such that $\nu_A e^{\nu_A} \le A$
        \STATE $\textrm{toggle} = 0$
        \FOR{$k = 0, 1,...,K$}
        \IF{$F(x_k) = 0$}
        \RETURN $x_k$
        \ENDIF
        \STATE $\lambda_k^0 = \lambda_{k-1}^0$ and $\lambda_k^1 = \lambda_{k-1}^1$
        \STATE $\gamma_k = \frac{\nu_A}{\lambda_k^0+\lambda_k^1\norm{F(x_k)}}$
        \STATE $\hx_k = x_k - \gamma_k F(x_k)$.
        \WHILE{$\| F(x_k) - F(\hx_k) \| > A \|F(x_k)\|$}
            \IF{$\textrm{toggle}=0$}
                \STATE 
                $\lambda_k^0=\lambda_k^0\left(\beta-\frac{\lambda_k^0}{\lambda_k^0+\lambda_k^1\|F(x_k)\|}\right)$ and $\textrm{toggle}=1$
            \ELSIF{$\textrm{toggle}=1$}
                \STATE
                $\lambda_k^1=\lambda_k^1\left(\beta-\frac{\lambda_k^1\|F(x_k)\|}{\lambda_k^0+\lambda_k^1\|F(x_k)\|}\right)$ and $\textrm{toggle}=0$
            \ENDIF
            \STATE $\gamma_k = \frac{\nu_A}{\lambda_k^0+\lambda_k^1\norm{F(x_k)}}$
            \STATE $\hx_k=x_k-\gamma_kF(x_k)$.
        \ENDWHILE
        \vspace{2mm}
        \STATE $\alpha_k = \frac{\la F(\hx_k), x_k - \hx_k \ra}{\|F(\hx_k)\|^2}$
        \vspace{2mm}
        \STATE $x_{k+1} = x_k - \alpha_k F(\hx_k)$
        \ENDFOR
    \end{algorithmic}
\end{algorithm}

\subsubsection{Sufficient choices of $\gamma_k$ for the critical condition}

\begin{proposition} 
\label{prop:gamma-critical-condition-sufficient}
Fix $A \in (0,1]$ and let $\hx_k=x_k-\gamma_kF(x_k)$. 
Assuming $F(x_k) \ne 0$, \eqref{eqn:gamma-critical-condition} is satisfied in each of the following cases:
\begin{itemize}[leftmargin=*]
   \setlength{\itemsep}{0pt}
    \item $F$ is $L$-Lipschitz (satisfies Definition~\ref{def:Lipschitz}) with constant
    $L>0$, and $0<\gamma_k\le \frac{A}{L}$.

     \item $F$ is $(L,\holderexponent)$-H\"older (satisfies \cref{def:Holder-continuity}), and $0 < \gamma_k \le \left(\frac{A}{L}\right)^{\frac{1}{\holderexponent}} \norm{F(x_k)}^{\frac{1-\holderexponent}{\holderexponent}}$.

    \item $F$ is $(L_0,L_1)$-Lipschitz (satisfies Definition~\ref{def:L0L1-Lipschitz}), and $0 < \gamma_k \le \frac{\nu_A}{L_0+L_1\norm{F(x_k)}}$, where $\nu_A>0$ satisfies $\nu_A e^{\nu_A} \le A$.

    \item $F$ is $L$-Lipschitz, $(L,\holderexponent)$-H\"older or $(L_0, L_1)$-Lipschitz and $\gamma_k$ is returned by the line-search rule in \cref{alg:PolyakEG_linesearch_unified}.
\end{itemize}
\end{proposition}

\begin{proof}
If $F$ is $(L,\holderexponent)$-H\"older and $0 < \gamma_k \le \left(\frac{A}{L}\right)^{\frac{1}{\holderexponent}} \norm{F(x_k)}^{\frac{1-\holderexponent}{\holderexponent}}$, then 
\begin{align*}
    \norm{F(\hx_k) - F(x_k)} \le L \norm{\hx_k - x_k}^\holderexponent = L \norm{\gamma_k F(x_k)}^\holderexponent \le L \norm{\left(\frac{A}{L}\right)^{\frac{1}{\holderexponent}} \norm{F(x_k)}^{\frac{1}{\holderexponent}}}^\holderexponent = A\norm{F(x_k)} .
\end{align*}
This argument covers the Lipschitz case $\holderexponent=1$ as well.
When $F$ is $(L_0, L_1)$-Lipschitz, using  
\eqref{eq:L0L1-alpha1-characterization} with $x=\hx_k$ and $y=x_k$, we obtain
\[
\begin{aligned}
    \norm{F(\hx_k)-F(x_k)}
    &\le
    \left(
        L_0+L_1\norm{F(x_k)}
    \right)
    \exp\left(
        L_1\norm{\hx_k-x_k}
    \right)
    \norm{\hx_k-x_k} \\
    &=
    \gamma_k
    \left(
        L_0+L_1\norm{F(x_k)}
    \right)
    \exp\left(
        L_1\gamma_k\norm{F(x_k)}
    \right)
    \norm{F(x_k)} \\
    & \le
    \nu_A \exp\left(
        \frac{\nu_A L_1\norm{F(x_k)}}
        {L_0+L_1\norm{F(x_k)}}
    \right)
    \norm{F(x_k)} \\
    &\le
    \nu_A e^{\nu_A}\norm{F(x_k)} \le A\norm{F(x_k)}
\end{aligned}
\]
where the third line uses $\gamma_k \le \frac{\nu_A}{L_0 + L_1 \norm{F(x_k)}}$.
Finally, the line-search in \cref{alg:PolyakEG_linesearch_unified} accepts $\gamma_k$ only if \eqref{eqn:gamma-critical-condition} is satisfied, so there is nothing to prove. 
We later show in the proof of \cref{theorem:L0L1-linesearch-deterministic} that the line-search always terminates under any one of the regularity assumptions considered.

\end{proof}

Having proved Proposition~\ref{prop:gamma-critical-condition-sufficient}, let us provide some remarks on its final expressions and how it is connected with the convergence guarantees of \cref{theorem:deterministic-master-theorem}. 

\begin{rem}[Role of the critical condition]
\label{rem:role-critical-condition}
\Cref{prop:gamma-critical-condition-sufficient} shows that the critical condition~\eqref{eqn:gamma-critical-condition} can be enforced by several standard choices of $\gamma_k$. 
This allows \cref{theorem:deterministic-master-theorem} below to be stated directly in terms of the critical condition, rather than in terms of a specific step-size rule. 
In this sense, the theorem is agnostic to the assumptions on the problem class or how $\gamma_k$ is selected, and is capable of stating the convergence results in a general, unified manner.
\end{rem}

\begin{rem}[Role of the parameter $A$]
\label{rem:role-parameter-A}
The parameter $A$ can be viewed a tolerance parameter; with larger $A$, a broader range of $\gamma_k$ is admissible. 
The most aggressive choice $A=1$ is allowed when one only needs to bound $\norm{F(\hx_k)}$.
However, we require $A\in (0,1)$ in order to control $F(x_k)$ using $F(\hx_k)$, because
\[
    \norm{F(\hx_k)} \ge \norm{F(x_k)}-\norm{F(\hx_k)-F(x_k)}
    \ge (1-A)\norm{F(x_k)} .
\]
Contraction in the strongly monotone case also requires $A\in (0,1)$.
\end{rem}

\begin{rem}[Choice of $\nu_A$ in the $(L_0,L_1)$ case]
\label{rem:choice-nuA}
For a fixed value of $A$, the largest possible value of $\nu_A$ satisfying $\nu_A e^{\nu_A}\le A$ is $\nu_A=W(A)$, 
where $W$ denotes the Lambert $W$ function. 
One may use any smaller positive value of $\nu_A$, e.g., $\nu_A = \frac{A}{1+A}$.
There is no universally optimal choice of $A$, and the best value
depends on which bound one wants to optimize.
\end{rem}

\subsection{Convergence of \hyperref[alg:PolyakEG]{\algname{PolyakEG}} Under Critical Condition}

In this section, we present a unified analysis of \hyperref[alg:PolyakEG]{\algname{PolyakEG}} under the critical condition~\eqref{eqn:gamma-critical-condition} and discuss its consequences in distinct settings.

\subsubsection{Unified convergence theorem}
\begin{theorem}
\label{theorem:deterministic-master-theorem}
Let $F$ be monotone. Suppose that we choose $\gamma_k > 0$ so that the critical condition \eqref{eqn:gamma-critical-condition} holds for all $k\ge 0$.
Then, \hyperref[alg:PolyakEG]{\algname{PolyakEG}} satisfies: 
\vspace{-2mm}
\begin{itemize}
    \item If \eqref{eqn:gamma-critical-condition} is satisfied with $A \in (0, 1]$, then the extrapolation points $\hat{x}_k$ satisfy:
    \begin{equation}
        \min_{k=0,\dots,K} \gamma_k^2 \sqnorm{F(\hat{x}_k)} \le \frac{(1+A)^2 \sqnorm{x_0 - x_*}}{K+1}. \label{eq:deterministic-master-theorem_eq1}
    \end{equation}
    \item If \eqref{eqn:gamma-critical-condition} is satisfied with $A \in (0, 1)$, then the iterates $x_k$ satisfy:
    \begin{equation}
        \min_{k=0,\dots,K} \gamma_k^2 \sqnorm{F(x_k)} \le \frac{(1+A) \sqnorm{x_0 - x_*}}{(1-A) (K+1)}. \label{eq:deterministic-master-theorem_eq2} 
    \end{equation}
    \item If $F$ is $\mu$-strongly monotone and $A \in (0, 1)$, then the iterates $x_k$ satisfy:
    \begin{equation}
        \sqnorm{x_{k+1} - x_*} \le \prod_{j=0}^{k} \left( 1 - \frac{2(1-A) \gamma_j \mu}{(1+A)^2} \right) \sqnorm{x_0 - x_*} . \label{eq:deterministic-master-theorem_eq3}
    \end{equation}
\end{itemize}
\end{theorem}

\begin{proof}
Let $F$ satisfy $\inprod{F(x) - F(y)}{x - y} \ge \mu \sqnorm{x-y}$ for any $x,y \in \reals^d$, where we take $\mu = 0$ if $F$ is merely monotone. Then we have
\begin{align*}
    \sqnorm{x_{k+1} - x_*} & = \sqnorm{x_k - \alpha_k F(\hat{x}_k) - x_*} \nonumber \\
    & = \sqnorm{x_k - x_*} - 2\alpha_k \inprod{F(\hat{x}_k)}{x_k - x_*} + \alpha_k^2 \sqnorm{F(\hat{x}_k)} \nonumber \\
    & = \sqnorm{x_k - x_*} - 2\alpha_k \inprod{F(\hat{x}_k)}{x_k - \hat{x}_k} - 2\alpha_k \inprod{F(\hat{x}_k)}{\hat{x}_k - x_*} + \alpha_k^2 \sqnorm{F(\hat{x}_k)} \nonumber \\
    & \le \sqnorm{x_k - x_*} - 2\alpha_k \inprod{F(\hat{x}_k)}{x_k - \hat{x}_k} - 2\alpha_k \mu \sqnorm{\hat{x}_k - x_*} + \alpha_k^2 \sqnorm{F(\hat{x}_k)} .
\end{align*}
Using Young's inequality (with $\omega > 0$ to be determined later), we obtain
\begin{align}
    & \sqnorm{x_{k+1} - x_*} \nonumber \\
    & \le \sqnorm{x_k - x_*} - 2\alpha_k \inprod{F(\hat{x}_k)}{x_k - \hat{x}_k} - \frac{2\alpha_k \mu}{1+\omega} \sqnorm{x_k - x_*} + \frac{2\alpha_k \mu}{\omega} \sqnorm{x_k - \hat{x}_k} + \alpha_k^2 \sqnorm{F(\hat{x}_k)} \nonumber \\
    & = \left( 1 - \frac{2\alpha_k \mu}{1+\omega} \right) \sqnorm{x_k - x_*} + \frac{2\alpha_k \mu}{\omega} \sqnorm{x_k - \hat{x}_k} - 2\alpha_k \inprod{F(\hat{x}_k)}{x_k - \hat{x}_k} + \alpha_k^2 \sqnorm{F(\hat{x}_k)} \nonumber \\
    & = \left( 1 - \frac{2\alpha_k \mu}{1+\omega} \right) \sqnorm{x_k - x_*} + \frac{2\alpha_k \gamma_k^2 \mu}{\omega} \sqnorm{F(x_k)} - \alpha_k \inprod{F(\hat{x}_k)}{x_k - \hat{x}_k} \nonumber \\
    & = \left( 1 - \frac{2\alpha_k \mu}{1+\omega} \right) \sqnorm{x_k - x_*} + \frac{2\alpha_k \gamma_k^2 \mu}{\omega} \sqnorm{F(x_k)} - \alpha_k \gamma_k \inprod{F(\hat{x}_k)}{F(x_k)} \nonumber \\
    & = \left( 1 - \frac{2\alpha_k \mu}{1+\omega} \right) \sqnorm{x_k - x_*} - \alpha_k \gamma_k \left( \inprod{F(\hat{x}_k)}{F(x_k)} - \frac{2\gamma_k \mu}{\omega} \sqnorm{F(x_k)} \right) .
    \label{eqn:LS-proof-bound}
\end{align}
Here, for the third line, we use the definition of $\alpha_k$ to replace $\alpha_k^2 \sqnorm{F(\hx_k)} = \alpha_k \inprod{F(\hx_k)}{x_k - \hx_k}$.

\vspace{.2cm}
\noindent
\textbf{$\bullet$ $F$ is monotone.} 
Substituting $\mu = 0$ into \eqref{eqn:LS-proof-bound} and using \Cref{lemma:alpha-lower-bound} we have
\begin{align}
    \sqnorm{x_{k+1} - x_*} & \le \sqnorm{x_k - x_*} - \alpha_k \gamma_k \inprod{F(\hat{x}_k)}{F(x_k)} \notag\\
    & \le \sqnorm{x_k - x_*} - \frac{\alpha_k \gamma_k}{2} \left( \sqnorm{F(\hat{x}_k)} + (1 - A^2) \sqnorm{F(x_k)} \right) \notag\\
    & \le \sqnorm{x_k - x_*} - \frac{\gamma_k^2}{2(1+A)} \left( \sqnorm{F(\hat{x}_k)} + (1 - A^2) \sqnorm{F(x_k)} \right) \label{eqn:monotone polyakeg bounded radius}.
\end{align}
Summing the above inequality for $k=0,1,\dots,K$ we obtain
\begin{align}
\label{eqn:monotone-bound-Fxk-Fhxk-combined}
    \sum_{k=0}^{K} \frac{\gamma_k^2}{2(1+A)}  \left( \sqnorm{F(\hat{x}_k)} + (1 - A^2) \sqnorm{F(x_k)} \right) \le \sqnorm{x_0 - x_*} - \sqnorm{x_{K+1} - x_*} \le \sqnorm{x_0 - x_*} .
\end{align}
Now note that we have
\begin{align}
\label{eqn:Fhxk-upper-bound-with-Fxk}
    \norm{F(\hx_k)} \le \norm{F(x_k)} + \norm{F(\hx_k) - F(x_k)} \le \norm{F(x_k)} + A\norm{F(x_k)} = (1+A) \norm{F(x_k)}
\end{align}
and
\begin{align}
\label{eqn:Fhxk-lower-bound-with-Fxk}
    \norm{F(\hx_k)} \ge \norm{F(x_k)} - \norm{F(\hx_k) - F(x_k)} \ge \norm{F(x_k)} - A\norm{F(x_k)} = (1-A) \norm{F(x_k)} .
\end{align}
Therefore, further lower bounding the left hand side of \eqref{eqn:monotone-bound-Fxk-Fhxk-combined} using \eqref{eqn:Fhxk-upper-bound-with-Fxk} we have
\begin{align*}
    \sqnorm{x_0 - x_*} & \ge \sum_{k=0}^K \frac{\gamma_k^2}{2(1+A)} \left( \sqnorm{F(\hx_k)} + \frac{1-A^2}{(1+A)^2} \sqnorm{F(\hx_k)} \right) \\
    & = \sum_{k=0}^K \frac{\gamma_k^2}{(1+A)^2} \sqnorm{F(\hx_k)} \\
    & \ge \frac{(K+1)}{(1+A)^2} \min_{k=0,\dots,K} \gamma_k^2 \sqnorm{F(\hx_k)} .
\end{align*}
This proves the first part. On the other hand, lower bounding \eqref{eqn:monotone-bound-Fxk-Fhxk-combined} in terms of $\norm{F(x_k)}$ using \eqref{eqn:Fhxk-lower-bound-with-Fxk} we have
\begin{align*}
    \sqnorm{x_0 - x_*} & \ge \sum_{k=0}^K \frac{\gamma_k^2}{2(1+A)} \left( (1-A)^2 \sqnorm{F(x_k)} + (1-A^2) \sqnorm{F(x_k)} \right) \\
    & = \sum_{k=0}^K \frac{1-A}{1+A} \gamma_k^2 \sqnorm{F(x_k)} \\
    & \ge \frac{(K+1) (1-A)}{1+A} \min_{k=0,\dots,K} \gamma_k^2 \sqnorm{F(x_k)} .
\end{align*}
This proves the second part of the monotone setting. 

\vspace{.2cm}
\noindent
\textbf{$\bullet$ $F$ is $\mu$-strongly monotone with $\mu>0$.}
Strong monotonicity of $F$ implies
\begin{align*}
    \inprod{x_k - \hat{x}_k}{F(x_k) - F(\hat{x}_k)} \ge \mu \sqnorm{x_k - \hat{x}_k} .
\end{align*}
Substituting $x_k - \hat{x}_k = \gamma_k F(x_k)$ in the above we get
\begin{align*}
    \inprod{F(x_k)}{F(x_k) - F(\hat{x}_k)} \ge \gamma_k \mu \sqnorm{F(x_k)} .
\end{align*}
We apply the above inequality to \eqref{eqn:LS-proof-bound}, which yields:
\begin{align*}
    & \sqnorm{x_{k+1} - x_*} \\
    & \le \left( 1 - \frac{2\alpha_k \mu}{1+\omega} \right) \sqnorm{x_k - x_*} - \alpha_k \gamma_k \left( \inprod{F(\hat{x}_k)}{F(x_k)} - \frac{2}{\omega} \gamma_k \mu \sqnorm{F(x_k)} \right) \\
    & \le \left( 1 - \frac{2\alpha_k \mu}{1+\omega} \right) \sqnorm{x_k - x_*} - \alpha_k \gamma_k \left( \inprod{F(\hat{x}_k)}{F(x_k)} - \frac{2}{\omega} \inprod{F(x_k)}{F(x_k) - F(\hat{x}_k)} \right) \\
    & = \left( 1 - \frac{2\alpha_k \mu}{1+\omega} \right) \sqnorm{x_k - x_*} - \alpha_k \gamma_k \left( \left(1 + \frac{2}{\omega}\right) \inprod{F(\hat{x}_k)}{F(x_k)} - \frac{2}{\omega} \sqnorm{F(x_k)} \right) .
\end{align*}
Now we use \Cref{lemma:alpha-lower-bound} to obtain
\begin{align*}
    \sqnorm{x_{k+1} - x_*} & \le \left( 1 - \frac{2\alpha_k \mu}{1+\omega} \right) \sqnorm{x_k - x_*} \\
    & - \alpha_k \gamma_k \left( \frac{1}{2} \left(1 + \frac{2}{\omega}\right) \left( \sqnorm{F(\hat{x}_k)} + (1 - A^2) \sqnorm{F(x_k)} \right) - \frac{2}{\omega} \sqnorm{F(x_k)} \right) \\
    & \le \left( 1 - \frac{2\alpha_k \mu}{1+\omega} \right) \sqnorm{x_k - x_*} \\
    & - \alpha_k \gamma_k \left( \frac{1}{2} \left(1 + \frac{2}{\omega}\right) \left( (1-A)^2 \sqnorm{F(x_k)} + (1 - A^2) \sqnorm{F(x_k)} \right) - \frac{2}{\omega} \sqnorm{F(x_k)} \right) \\
    & \le \left( 1 - \frac{2\alpha_k \mu}{1+\omega} \right) \sqnorm{x_k - x_*} - \alpha_k \gamma_k \left( 1 - A - \frac{2A}{\omega} \right) \sqnorm{F(x_k)} .
\end{align*}
where the second inequality uses \eqref{eqn:Fhxk-lower-bound-with-Fxk}.
Now taking $\omega = \frac{2A}{1-A}$ in the last inequality gives
\begin{align*}
    \sqnorm{x_{k+1} - x_*} & \le \left( 1 - \frac{2(1-A)}{1+A} \alpha_k \mu \right) \sqnorm{x_k - x_*} \le \left( 1 - \frac{2(1-A)}{(1+A)^2} \gamma_k \mu \right) \sqnorm{x_k - x_*} .
\end{align*}
where for the second line we use \cref{lemma:alpha-lower-bound}.
Unrolling the recursion gives the desired result.
\end{proof}

Below, we discuss which specific convergence rates \cref{theorem:deterministic-master-theorem} implies for different settings and choices of $\gamma_k$.

\subsubsection{Convergence results implied by the unified theorem}

\paragraph{Constant $\gamma_k$ for Lipschitz problems.}
Assuming $F$ is $L$-Lipschitz, taking $A=1$ and $\gamma_k = \frac{1}{L}$, we can rewrite \eqref{eq:deterministic-master-theorem_eq1} as $\textstyle \min_{0 \leq k \leq K} \sqnorm{F(\hx_k)} \leq \frac{4 L^2\|x_0 - x_*\|^2}{K+1}$. 
This recovers the result of \citet[Corollary~3.2]{pethickEscapingLimitCycles2022} for the unconstrained case. 
Moreover, with $A \in (0, 1)$ and $\gamma_k = \frac{A}{L}$, \eqref{eq:deterministic-master-theorem_eq2} reads as $\min_{0 \leq k \leq K} \| F(x_k)\|^2 \leq \frac{(1 + A) L^2 \|x_0 - x_*\|^2}{A^2 (1 - A) (K+1)}$.
Unlike the case $A=1$, where convergence is stated only in terms of $\hx_k$ as in \citet{pethickEscapingLimitCycles2022}, this ensures that one only needs to focus on the $x_k$ sequence, regardless of the problem being strongly monotone or merely monotone.

The linear convergence of \hyperref[alg:PolyakEG]{\algname{PolyakEG}} for strongly monotone $F$ is classical; it was shown in \citet{solodovModifiedProjectiontypeMethods1996}.
Nevertheless, we refine it into an explicit, tight, quantitative rate. 
In \eqref{eq:deterministic-master-theorem_eq3}, for $A \in (0,1)$ we can take $\gamma_k = \frac{A}{L}$ and get $\sqnorm{x_{k+1} - x_*} \le \left(1 - \frac{2A(1-A)}{(1+A)^2} \frac{\mu}{L}\right)^{k+1} \sqnorm{x_0 - x_*}$.
With $A=\frac{1}{3}$ and $\gamma_k = \frac{1}{3L}$, the linear factor becomes $1 - \nicefrac{\mu}{4L}$, matching the state-of-the-art rate for \hyperref[eq:EG]{\algname{EG}} \citep{MokhtariOzdaglarPattathil2020_unified, AzizianMitliagkasLacoste-JulienGidel2020_tight} that uses $\alpha_k = \gamma_k = \frac{1}{4L}$. 
It is perhaps interesting that \algname{PolyakEG} can achieve the same rate with $\gamma_k$ exceeding $\frac{1}{4L}$; it also demonstrates a technically novel aspect of our analysis.

\paragraph{The case of H\"older-continuous problems.}
Taking $A \in (0,1)$ and $\gamma_k = \left(\frac{A}{L}\right)^{\frac{1}{\holderexponent}} \norm{F(x_k)}^{\frac{1-\holderexponent}{\holderexponent}}$, we can rewrite \eqref{eq:deterministic-master-theorem_eq2} as $\min_{0 \le k \le K} \left(\frac{A\norm{F(x_k)}}{L} \right)^{2/\holderexponent} \le \frac{(1+A)\sqnorm{x_0 - x_*}}{(1-A) (K+1)}$. 
This yields the computation complexity of $\cO\left(\epsilon^{-2/\holderexponent}\right)$ for finding a point satisfying $\norm{F(\cdot)} \le \epsilon$. 
In the unconstrained Euclidean setting, this matches the complexity of \citet[Theorem~4.4(a)]{dangConvergencePropertiesNonEuclidean2015}.

\paragraph{\textbf{The case of $(L_0,L_1)$-Lipschitz problems.}}
Given $A \in(0,1)$ and $\nu_A e^{\nu_A}\le A$, consider $\gamma_k = \frac{\nu_A}{L_0 + L_1 \norm{F(x_k)}}$.
Then \eqref{eq:deterministic-master-theorem_eq2} yields
\begin{align}
    & \min_{0\le k\le K} \frac{\nu_A^2\|F(x_k)\|^2}{(L_0+L_1\|F(x_k)\|)^2} \le \frac{(1+A)\|x_0-x_*\|^2}{(1-A)(K+1)} = \tau_K^2 \implies 
    \min_{0\le k\le K} \norm{F(x_k)} \le \frac{L_0 \tau_K}{\nu_A - L_1 \tau_K}
    \label{eqn:l0-l1-monotone-result}
\end{align}
for sufficiently large $K$ where $\tau_K := \sqrt{\frac{1+A}{1-A}} \frac{\norm{x_0 - x_\star}}{\sqrt{K+1}} = \cO\left(\frac{1}{\sqrt{K}}\right)$, indicating an eventual sublinear convergence.
When $F$ is additionally $\mu$-strongly monotone, \eqref{eq:deterministic-master-theorem_eq3} yields
\begin{equation}
    \|x_{k+1}-x_*\|^2 \le \prod_{j=0}^k \left( 1- \frac{2(1-A)\nu_A\mu} {(1+A)^2(L_0+L_1\|F(x_j)\|)} \right) \|x_0-x_*\|^2 . \label{eqn:l0-l1-usm-result}
\end{equation}
This implies $\norm{x_k - x_*} \le \norm{x_0 - x_*}$ for all $k\ge 0$ and thus by \eqref{eq:L0L1-alpha1-characterization}
\begin{align*}
    \|F(x_k)\| = \|F(x_k)-F(x_*)\| & \leq L_0\exp(L_1\norm{x_k-x_*})\norm {x_k-x_*} \\ 
    & \leq L_0\exp(L_1\norm {x_0-x_*})\norm {x_0-x_*}
\end{align*}
and combining this with \eqref{eqn:l0-l1-usm-result} yields linear convergence with a fixed factor. 
This result and \eqref{eqn:l0-l1-monotone-result} are similar to the convergence guarantees for \algname{EG} with $\alpha_k = \gamma_k$ from \cite{vankovGeneralizedSmoothVariational2024,choudhury2025extragradient} for $(L_0, L_1)$-Lipschitz root finding problems.

\paragraph{Complexity of \algname{PolyakEG-LS}.}
To understand the computational complexity of \algname{PolyakEG-LS}, we need to bound the total number of line-search loops and the resulting step-sizes as below.

\begin{proposition}
\label{theorem:L0L1-linesearch-deterministic}
If $F$ is $L$-Lipschitz, $(L,\holderexponent)$-H\"older or $(L_0, L_1)$-Lipschitz, then at each iteration $k$, \cref{alg:PolyakEG_linesearch_unified} accepts a step-size $\gamma_k$ satisfying the critical condition \eqref{eqn:gamma-critical-condition}.
Furthermore, 
\begin{itemize}
    \item[(a)]
    When $F$ is $L$-Lipschitz or $(L_0, L_1)$-Lipschitz, then the total number of while-loop calls over all iterations is at most
    \begin{equation}
    \label{eqn:linesearch-operator-call-bound}
        N
        \eqdef
        2\max\left\{
            0,
            \left\lceil
                \frac{\log(L_0/\lambda_{-1}^0)}
                     {\log(\beta-1)}
            \right\rceil,
            \left\lceil
                \frac{\log(L_1/\lambda_{-1}^1)}
                     {\log(\beta-1)}
            \right\rceil
        \right\} 
    \end{equation}
    where we disregard the last term involving $L_1/\lambda_{-1}^1$ when $\lambda_{-1}^1 = L_1 = 0$, and we require $\lambda_{-1}^1 > 0$ when $L_1 > 0$.
    Consequently, for all accepted steps we have
    \[
        \lambda_k^0
        \le
        \bar\lambda^0
        \eqdef
        \beta^{N/2}\lambda_{-1}^0,
        \qquad
        \lambda_k^1
        \le
        \bar\lambda^1
        \eqdef
        \beta^{N/2}\lambda_{-1}^1.
    \]

    \item[(b)] Suppose $F$ is $(L,\holderexponent)$-H\"older with $\holderexponent \in (0,1)$. Then for any $\epsilon > 0$, the number of while-loop calls $N_\epsilon$ before reaching $\norm{F(x_k)} \le \epsilon$ for the first time is bounded by $\cO\left( \log \frac{1}{\epsilon} \right)$, and the step-size satisfies $\gamma_k \ge \min \left\{ \frac{1}{\beta - 1} \left(\frac{A}{L}\right)^{\frac{1}{\holderexponent}} \epsilon^{\frac{1-\holderexponent}{\holderexponent}} , \frac{\nu_A}{\lambda_{-1}^0} \right\}$.
\end{itemize}
\end{proposition}

\noindent
The above result, combined with \cref{theorem:deterministic-master-theorem}, shows that the complexity of line-search is dominated by the number of iterations needed to attain $\epsilon$-residual.
This is clear for the case of $L$-Lipschitz or $(L_0, L_1)$-Lipschitz problems because the number of total line-search loops is finite.
In the $(L,\holderexponent)$-H\"older case, if $\norm{F(x_k)} > \epsilon$ for all $k=0,\dots,K$, then combining \cref{theorem:L0L1-linesearch-deterministic}(b) with \cref{theorem:deterministic-master-theorem} we obtain
\[
    \frac{(1+A)\sqnorm{x_0 - x_*}}{(1-A)(K+1)} \ge \min_{0\le k \le K} \gamma_k^2 \sqnorm{F(x_k)} \ge \min \left\{ \left( \frac{\nu_A \left( \frac{A\epsilon}{L} \right)^\frac{1}{\holderexponent}}{\beta - 1}\right)^2 , \left( \frac{\nu_A \epsilon}{\lambda_{-1}^0} \right)^2 \right\} ,
\]
which yields an $\cO\left(\epsilon^{-2/\alpha} \right)$ upper bound on $K$. 
On the other hand, the number of line-search loops up to iteration $K$ is $\cO\left(\log \epsilon^{-1}\right)$, so the order of total complexity for finding a point with $\norm{F(\cdot)} \le \epsilon$ remains $\cO\left(\epsilon^{-2/\alpha} \right)$.
Therefore, the theoretical complexity of \algname{PolyakEG-LS} is comparable to the versions that select $\gamma_k$ based on the knowledge of problem-dependent parameters such as $L,\holderexponent,L_0$ and $L_1$, despite the algorithm being parameter-free.

\begin{proof}[Proof of \cref{theorem:L0L1-linesearch-deterministic}]
(a) If $F$ is $(L_0, L_1)$-Lipschitz, then once we reach $\lambda_k^0 \ge L_0$ and $\lambda_k^1 \ge L_1$, we have \eqref{eqn:gamma-critical-condition} satisfied for any $x_k$ by \eqref{eq:L0L1-alpha1-characterization}, and no further line-search loops are called.
At each while loop call, we alternatingly increase either $\lambda_k^0$ or $\lambda_k^1$ by the factor
\begin{align*}
    \beta - \frac{\lambda_k^0}{\lambda_k^0 + \lambda_k^1 \norm{F(x_k)}} \ge \beta - 1 > 1 \quad \text{or} \quad \beta - \frac{\lambda_k^1 \norm{F(x_k)}}{\lambda_k^0 + \lambda_k^1 \norm{F(x_k)}} \ge \beta - 1 > 1 ,
\end{align*}
respectively. 
The number of total while loop calls throughout the algorithm's execution is therefore upper bounded by \eqref{eqn:linesearch-operator-call-bound}.
Both $\lambda_k^0$ or $\lambda_k^1$ is increased at most $\frac{N}{2}$ times, by a factor at most $\beta$, so we have $\lambda_k^i \le \beta^{N/2} \lambda_{-1}^i$ for $i=0,1$.\medskip

\noindent
(b) It remains to show that the while loop run always terminates for $(L,\holderexponent)$-H\"older problems with $\holderexponent \in (0,1)$, and the corresponding bound on the number of while-loop calls in this case.
If $F(x_k) = 0$ then $x_k$ is already a solution, and the algorithm terminates.
Otherwise, by \cref{prop:gamma-critical-condition-sufficient} we know that \eqref{eqn:gamma-critical-condition} is satisfied if $\lambda_k^0 \ge \nu_A \left(\frac{L}{A}\right)^{\frac{1}{\holderexponent}} \norm{F(x_k)}^{-\frac{1-\holderexponent}{\holderexponent}}$, so the line-search at iteration $k$ terminates  in finite steps.
Additionally, because each rejected line-search loop increases $\lambda_k^0$ by factor $\beta-1$, unless  the initial $\lambda_k^0$ already satisfies the stopping criterion, 
the accepted $\lambda_k^0$ is at most $(\beta-1) \nu_A \left(\frac{L}{A}\right)^{\frac{1}{\holderexponent}} \norm{F(x_k)}^{-\frac{1-\holderexponent}{\holderexponent}}$. 
Now suppose that $\epsilon$-accuracy is not reached until the $k$-th iteration, i.e., $\norm{F(x_j)} > \epsilon$ for $j=0,\dots,k$.
This implies $\lambda_k^0 \le \max \left\{ \lambda_{-1}^0 , (\beta-1) \nu_A \left(\frac{L}{A}\right)^{\frac{1}{\holderexponent}} \epsilon^{-\frac{1-\holderexponent}{\holderexponent}} \right\}$, and thus,
\[
    \gamma_k = \frac{\nu_A}{\lambda_k^0} \ge \min \left\{ \frac{1}{\beta - 1} \left(\frac{A}{L}\right)^{\frac{1}{\holderexponent}} \epsilon^{\frac{1-\holderexponent}{\holderexponent}} , \frac{\nu_A}{\lambda_{-1}^0} \right\}
\]
and the number of line-search rejections up to iteration $k$ is at most
\[
    N_\epsilon = \log_{\beta-1} \left( \max\left\{ 1, \frac{(\beta-1) \nu_A \left(\frac{L}{A}\right)^{\frac{1}{\holderexponent}} \epsilon^{-\frac{1-\holderexponent}{\holderexponent}}}{\lambda_{-1}^0} \right\} \right) = \left[ C + \frac{(1-\holderexponent)}{\holderexponent \log (\beta - 1)} \log \frac{1}{\epsilon} \right]_+ = \cO \left( \log  \frac{1}{\epsilon} \right) 
\]
where $C$ is a constant depending only on $L, \holderexponent, A, \nu_A, \beta$ and $\lambda_{-1}^0$.

\end{proof}

\section{Stochastic Setting}\label{sec:stochastic}
Stochastic Polyak step-size methods have recently received significant attention in stochastic convex minimization. In particular, \citet{loizouStochasticPolyakStepsize2021} showed that stochastic gradient methods with Polyak-type step-sizes can achieve strong convergence guarantees without requiring knowledge of problem-dependent constants, provided that the problem satisfies the interpolation condition, meaning that all component functions share a common minimizer. 
We consider the analogous interpolation condition for the stochastic root-finding setting:
there exists $x_* \in \R^d$ such that $F_i(x_*)=0$ for every $i\in[n]$. 
When this is the case, we refer to the problem as \emph{interpolated}.
The interpolation assumption is natural in overparameterized learning but restrictive in general stochastic optimization settings. Subsequently, \citet{orvietoDynamicsSGDStochastic2022} studied decreasing variants of the stochastic Polyak step-size that remove the interpolation requirement and handle a broader class of convex objectives, albeit under the assumption that the algorithm iterates remain bounded.

Motivated by these developments, we provide analogous guarantees for stochastic extensions of \algname{PolyakEG}. 
We first consider the direct stochastic analogue of \algname{PolyakEG}, named \algname{PolyakSEG} (\Cref{alg:PolyakSEG}), which replaces $F(x_k)$ and $F(\hx_k)$ by stochastic estimates $F_{\mathcal{S}_k}(x_k)$ and $F_{\mathcal{S}_k}(\hx_k)$. 
As in the stochastic minimization results of \citet{loizouStochasticPolyakStepsize2021}, this naive extension naturally leads to convergence guarantees under the interpolation condition. 
To go beyond this restrictive regime, we then introduce and analyze \textit{Decreasing Polyak} \algname{SEG} (\algname{DecPolyakSEG}; \Cref{alg:decPolyakSEG}), which adapts the decreasing-Polyak philosophy of \citet{orvietoDynamicsSGDStochastic2022} to stochastic extragradient methods for monotone root-finding problems.

\paragraph{\textbf{Connection to Minimization.}}

For stochastic convex minimization setting: $\underset{x\in \reals^d}{\text{minimize}}\,\, f(x) = \frac{1}{n} \sum_{i = 1}^n f_i(x)$, \citet{loizouStochasticPolyakStepsize2021} and \citet{orvietoDynamicsSGDStochastic2022} respectively studied stochastic gradient descent (\algname{SGD}) $x_{k+1} - x_k - \eta_k \nabla f_{\mathcal{S}_k}(x_k)$ with the Stochastic Polyak step-size (\algname{SPS}) $\eta_k = \frac{f_{\mathcal{S}_k}(x_k) - \min_x f_{\mathcal{S}_k}(x)}{\| \nabla f_{\mathcal{S}_k}(x_k)\|^2}$ \big(here $f_{\mathcal{S}_k}(x) = \frac{1}{B} \sum_{i \in \mathcal{S}_k} f_i(x)$\big), and its decreasing variant \algname{DecSPS} $$\eta_k = \frac{1}{c_{k+1}} \min \left\{ \frac{f_{\mathcal{S}_k}(x_k) - \min_x f_{\mathcal{S}_k}(x)}{c \| \nabla f_{\mathcal{S}_k}(x_k)\|^2}, c_k \eta_{k-1}\right\}$$ 
where $c_k > 0$ is a nondecreasing sequence.
Specifically, \citet{loizouStochasticPolyakStepsize2021} showed that \algname{SPS} converges linearly given that each $f_i$ is smooth and strongly convex and $x_*$ is an interpolating solution.
Without interpolation, however, only the convergence to a neighborhood could be guaranteed.
\citet{orvietoDynamicsSGDStochastic2022} introduced \algname{DecSPS} to resolve this issue and proved its (sublinear) convergence without assuming an interpolated solution.
Our \algname{PolyakSEG} and \algname{DecPolyakSEG}, introduced below, can be respectively viewed as analogues of \algname{SPS} and \algname{DecSPS} (see \cref{tab:polyak_step_stochastic}).

\begin{algorithm}[H]
\small
    \caption{\algname{\textcolor{PineGreen}{PolyakSEG}}}
    \label{alg:PolyakSEG}
    \begin{algorithmic}[1]
        \REQUIRE Initial point $x_0 \in \R^d$ and a rule for choosing $\{\gamma_k\}_{k=0}^\infty$.
        \FOR{$k = 0, 1,...,K$}
        \STATE Sample $\mathcal{S}_k \subseteq [n]$.
        \STATE $\hx_k = x_k - \gamma_k F_{\mathcal{S}_k}(x_k)$.
        \STATE \colorbox{green!20}{$\alpha_k = \frac{\la F_{\mathcal{S}_k}(\hx_k), x_k - \hx_k \ra}{\|F_{\mathcal{S}_k}(\hx_k)\|^2}$.}
        \vspace{2mm}
        \STATE $x_{k+1} = x_k - \alpha_k F_{\mathcal{S}_k}(\hx_k)$.
        \ENDFOR
    \end{algorithmic}
\end{algorithm}

\begin{algorithm}[H]
\small
    \caption{\algname{\textcolor{PineGreen}{DecPolyakSEG}}}
    \label{alg:decPolyakSEG}
    \begin{algorithmic}[1]
        \REQUIRE Initial point $x_0 \in \R^d$, a non-decreasing positive sequence $\{ c_k\}_{k=-1}^{\infty}$, and a rule for choosing $\{\gamma_k\}_{k=-1}^\infty$ such that $0 < \gamma_k \le \frac{c_{k-1}}{c_k}\gamma_{k-1}$ for $k\ge 0$.
        \STATE $\alpha_{-1} = \infty$.
        \FOR{$k = 0, 1,...,K$}
        \STATE Sample $\mathcal{S}_k \subseteq [n]$.
        \STATE $\hx_k = x_k - \gamma_k F_{\mathcal{S}_k}(x_k)$.
        \STATE \colorbox{green!20}{$\alpha_k = \min \left\{ \frac{\la F_{\mathcal{S}_k}(\hx_k), x_k - \hx_k \ra}{\|F_{\mathcal{S}_k}(\hx_k)\|^2}, \alpha_{k-1}\right\}$.}
        \vspace{2mm}
        \STATE $x_{k+1} = x_k - \alpha_k F_{\mathcal{S}_k}(\hx_k)$.
        \ENDFOR
    \end{algorithmic}
\end{algorithm}

\paragraph{Avoiding division by zero.}
For deterministic \algname{PolyakEG}, $F(\hx_k) = 0$ implies that $\hx_k$ is a solution, and we can terminate the algorithm before computing $\alpha_k$.
In the stochastic setting, however, $F_{\cS_k}(\hx_k) = 0$ does not necessarily imply that $\hx_k$ is a solution and we cannot simply terminate.
To avoid division by zero in this case, we may set $\alpha_k = \min\{\alpha_{k-1}, \gamma_k\}$ in both Algorithms~\ref{alg:PolyakSEG} and \ref{alg:decPolyakSEG}.
In either case, $x_{k+1} = x_k$ since $F_{\cS_k}(\hx_k) = 0$, so the $k$-th iteration produces no update.
All subsequent convergence arguments remain unchanged under this convention.

\paragraph{Stochastic critical condition.} 
We introduce the following critical condition, whose role is similar as in the deterministic analysis, but with the difference that we impose the condition on the sample operator $F_{\cS_k}$.

\begin{definition} 
\label{def:critical-condition-stochastic}
We say $\gamma_k$ satisfies the critical condition at iteration $k$ with respect to $F_{\mathcal{S}_k}$ for some $A \in (0, 1]$ if
\begin{equation}\label{eqn:gamma-critical-condition-stochastic}
\textstyle
    \norm{F_{\mathcal{S}_k}(\hx_k) - F_{\mathcal{S}_k}(x_k)} \le A \norm{F_{\mathcal{S}_k}(x_k)}
\end{equation}
for $\hx_k = x_k - \gamma_k F_{\mathcal{S}_k}(x_k)$.
\end{definition}

As before, we can enforce this condition by choosing $\gamma_k$ according to theoretical assumptions.
In particular, we will assume the uniform sample-wise $L$-Lipschitzness throughout this section for simplicity of the analysis, and in this case, \eqref{eqn:gamma-critical-condition-stochastic} will be satisfied with any $0 < \gamma_k \le \frac{A}{L}$.
However, this does not imply that stochastic \algname{PolyakEG} methods are intrinsically limited to uniformly Lipschitz problems.
We believe broader settings, such as non-uniformly Lipschitz or $(L_0, L_1)$-Lipschitz problems, can be handled, e.g., by incorporating more flexible line-search schemes, which we do not formally pursue in this work.

Now we state some consequences of the stochastic critical condition~\eqref{eqn:gamma-critical-condition-stochastic} which would be useful for all the subsequent convergence analyses. 

\begin{lemma}
\label{lemma:PolyakSEG-LS-key-properties}
Let $x_k$ be an iterate from either \hyperref[alg:PolyakSEG]{\algname{PolyakSEG}} or \hyperref[alg:decPolyakSEG]{\algname{DecPolyakSEG}}, and suppose that $\gamma_k$ satisfies the critical condition (\Cref{def:critical-condition-stochastic}). 
Then, for $k=0,1,\dots$,
\begin{enumerate}
[label=(\alph*)]
    \item $(1-A) \norm{F_{\mathcal{S}_k}(x_k)} \le \norm{F_{\mathcal{S}_k}(\hx_k)} \le (1+A)\norm{F_{\mathcal{S}_k}(x_k)}$
    \item $\inprod{F_{\mathcal{S}_k}(\hx_k)}{F_{\mathcal{S}_k}(x_k)}
    \ge \max \left\{ (1-A) \sqnorm{F_{\mathcal{S}_k}(x_k)} , \frac{\sqnorm{F_{\mathcal{S}_k}(\hx_k)}}{1+A} \right\}$ if $0 < A < 1$
    \item $\frac{\gamma_k}{1+A} \le \alpha_k \le \frac{\gamma_k}{1-A}$ if $0 < A < 1$
\end{enumerate}
holds almost surely.
\end{lemma}

\begin{proof}

By the critical condition \eqref{eqn:gamma-critical-condition-stochastic}, we obtain
\begin{align*}
    \norm{F_{\mathcal{S}_k}(\hx_k)} \le \norm{F_{\mathcal{S}_k}(x_k)} + \norm{F_{\mathcal{S}_k}(\hx_k) - F_{\mathcal{S}_k}(x_k)} \le \norm{F_{\mathcal{S}_k}(x_k)} + A\norm{F_{\mathcal{S}_k}(x_k)} = (1+A) \norm{F_{\mathcal{S}_k}(x_k)}
\end{align*}
and
\begin{align*}
    \norm{F_{\mathcal{S}_k}(\hx_k)} \ge \norm{F_{\mathcal{S}_k}(x_k)} - \norm{F_{\mathcal{S}_k}(\hx_k) - F_{\mathcal{S}_k}(x_k)} \ge \norm{F_{\mathcal{S}_k}(x_k)} - A\norm{F_{\mathcal{S}_k}(x_k)} \ge (1-A) \norm{F_{\mathcal{S}_k}(x_k)}
\end{align*}
which proves (a).
Next, observe that
\begin{align*}
    \inprod{F_{\mathcal{S}_k}(\hx_k)}{F_{\mathcal{S}_k}(x_k)} & = \frac{1}{2} \left( \sqnorm{F_{\mathcal{S}_k}(\hx_k)} - \sqnorm{F_{\mathcal{S}_k}(\hx_k) - F_{\mathcal{S}_k}(x_k)} + \sqnorm{F_{\mathcal{S}_k}(x_k)} \right) \\
    & \ge \frac{1}{2} \left( \sqnorm{F_{\mathcal{S}_k}(\hx_k)} - A^2 \sqnorm{F_{\mathcal{S}_k}(x_k)} +\sqnorm{F_{\mathcal{S}_k}(x_k)} \right) \\
    & = \frac{1}{2} \left( \sqnorm{F_{\mathcal{S}_k}(\hx_k)} + (1 - A^2) \sqnorm{F_{\mathcal{S}_k}(x_k)} \right) 
\end{align*}
where the inequality uses \eqref{eqn:gamma-critical-condition-stochastic}.
Applying the first inequality of (a) to the last expression, we obtain
\begin{align*}
    \inprod{F_{\mathcal{S}_k}(\hx_k)}{F_{\mathcal{S}_k}(x_k)} \ge \frac{1}{2} \left( (1-A)^2 \sqnorm{F_{\mathcal{S}_k}(x_k)} + (1 - A^2) \sqnorm{F_{\mathcal{S}_k}(x_k)} \right) = (1-A) \sqnorm{F_{\mathcal{S}_k}(x_k)} ,
\end{align*}
while applying the second inequality of (a) gives
\begin{align*}
    \inprod{F_{\mathcal{S}_k}(\hx_k)}{F_{\mathcal{S}_k}(x_k)} \ge \frac{1}{2} \left( \sqnorm{F_{\mathcal{S}_k}(\hx_k)} + (1 - A^2) \frac{1}{(1+A)^2} \sqnorm{F_{\mathcal{S}_k}(\hx_k)} \right) = \frac{\sqnorm{F_{\mathcal{S}_k}(\hx_k)}}{1+A} .
\end{align*}
This proves (b).
Finally, for \algname{PolyakSEG} we immediately obtain
\begin{align*}
    \alpha_k = \frac{\inprod{F_{\mathcal{S}_k}(\hx_k)}{x_k - \hx_k}}{\sqnorm{F_{\mathcal{S}_k}(\hx_k)}} = \frac{\gamma_k \inprod{F_{\mathcal{S}_k}(\hx_k)}{F_{\mathcal{S}_k}(x_k)}}{\sqnorm{F_{\mathcal{S}_k}(\hx_k)}} \ge \frac{\gamma_k}{1+A} .
\end{align*}
For the case of \algname{DecPolyakSEG}, we have $\gamma_k \le \gamma_{k-1}$ and $\alpha_k \le \alpha_{k-1}$ for all $k=0,1,\dots$ by construction.
As $\alpha_{-1} = \infty$, 
\begin{align*}
    \alpha_0 = \min \left\{ \frac{\inprod{F_{\cS_0}(\hx_0)}{x_0 - \hx_0}}{\sqnorm{F_{\cS_0}(\hx_0)}} , \alpha_{-1} \right\} = \frac{\inprod{F_{\cS_0}(\hx_0)}{x_0 - \hx_0}}{\sqnorm{F_{\cS_0}(\hx_0)}} \ge \frac{\gamma_0}{1+A} .
\end{align*}
Now we use induction on $k=1,2,\dots$: assuming that $\alpha_{k-1} \ge \frac{\gamma_{k-1}}{1+A}$, we have $\alpha_{k-1} \ge \frac{\gamma_{k}}{1+A}$ (since $\gamma_{k} \le \gamma_{k-1}$). 
We further have $\frac{\inprod{F_{\mathcal{S}_k}(\hx_k)}{x_k - \hx_k}}{\sqnorm{F_{\mathcal{S}_k}(\hx_k)}} \ge \frac{\gamma_k}{1+A}$ by (b), which implies
\begin{align*}
    \alpha_k = \min \left\{ \alpha_{k-1}, \frac{\inprod{F_{\mathcal{S}_k}(\hx_k)}{x_k - \hx_k}}{\sqnorm{F_{\mathcal{S}_k}(\hx_k)}} \right\} \ge \frac{\gamma_k}{1+A} ,
\end{align*}
completing the induction. This proves the first inequality in (c).
Finally, for the second inequality in (c), observe that for both \algname{PolyakSEG} and \algname{DecPolyakSEG},
\begin{align*}
    \alpha_k & \le \frac{\gamma_k \inprod{F_{\mathcal{S}_k}(\hx_k)}{F_{\cS_k}(x_k)}}{\sqnorm{F_{\mathcal{S}_k}(\hx_k)}} \le \frac{\gamma_k}{2} \frac{\frac{1}{1-A} \sqnorm{F_{\mathcal{S}_k}(\hx_k)} + (1-A) \sqnorm{F_{\cS_k}(x_k)}}{\sqnorm{F_{\mathcal{S}_k}(\hx_k)}} \\
    & \le \frac{\gamma_k}{2} \frac{\frac{1}{1-A} \sqnorm{F_{\mathcal{S}_k}(\hx_k)} + \frac{1}{1-A} \sqnorm{F_{\mathcal{S}_k}(\hx_k)}}{\sqnorm{F_{\mathcal{S}_k}(\hx_k)}} = \frac{\gamma_k}{1-A}
\end{align*}
where the third inequality uses (a).
\end{proof}

\subsection{\algname{PolyakSEG}: Convergence for Interpolated Problems}\label{subsec:PolyakSEG}

In this section we consider \algname{PolyakSEG} (\Cref{alg:PolyakSEG}), which is an immediate extension of \algname{PolyakEG} where $F$ is replaced by $F_{\cS_k}$, where $\cS_k$ is a mini-batch sampled at each iteration. 
For interpolated problems, the essentially same descent argument for deterministic \algname{PolyakEG} can be applied samplewisely, yielding the following result.

\begin{theorem}\label{theorem:PolyakSEG-convergence1}
Suppose that each $F_i \colon \R^d\to\R^d$ is monotone and there exists an interpolating solution $x_*$ satisfying $F_i(x_*)=0$ almost surely.
Suppose that we choose $\gamma_k > 0$ so that \eqref{eqn:gamma-critical-condition-stochastic} holds with $A \in (0,1)$.
Then \algname{PolyakSEG} satisfies:
\begin{itemize}
    \item Almost surely, 
    \begin{equation}
    \label{eq:PolyakSEG-sampled-residual-rate}
        \frac{1}{K+1} \sum_{k=0}^K \gamma_k^2 \sqnorm{F_{\cS_k}(x_k)}
        \leq
        \frac{(1+A)\|x_0-x_*\|^2}
        {(1-A)(K+1)}.
    \end{equation}

    \item If, in addition, there exist constants $\mu_\cS \geq 0$ such that $\inprod{F_\cS(x) - F_\cS(y)}{x-y} \geq \mu_\cS\|x-y\|^2$ for all $x,y\in\R^d$ and mini-batches $\cS$, then
    \begin{equation}
    \label{eq:PolyakSEG-heterogeneous-expectation}
        \Expk{\|x_{k+1}-x_*\|^2}
        \leq
        \left(
            1-
            \frac{2(1-A)}{(1+A)^2}
            \Expk{\gamma_k\mu_{\mathcal S_k}}
        \right)
        \|x_k-x_*\|^2
    \end{equation}
    where $\Expk{\cdot}$ denotes the conditional expectation with respect to the randomness revealed before drawing $\cS_k$.

\end{itemize}
\end{theorem}

\begin{proof}

Because $F_i(x_*) = 0$ almost surely, we have $F_{\cS_k}(x_*) = 0$, and as in the deterministic case,
\begin{align}
    \sqnorm{x_{k+1} - x_*} & = \sqnorm{x_k - \alpha_k F_{\mathcal{S}_k}(\hx_k) - x_*} \nonumber \\
    & = \sqnorm{x_k - x_*} - 2\alpha_k \inprod{F_{\mathcal{S}_k}(\hx_k)}{x_k - x_*} + \alpha_k^2 \sqnorm{F_{\mathcal{S}_k}(\hx_k)} \nonumber \\
    & = \sqnorm{x_k - x_*} - 2\alpha_k \inprod{F_{\mathcal{S}_k}(\hx_k)}{x_k - \hat{x}_k} - 2\alpha_k \inprod{F_{\mathcal{S}_k}(\hx_k)}{\hat{x}_k - x_*} + \alpha_k^2 \sqnorm{F_{\mathcal{S}_k}(\hx_k)} \nonumber \\
    & = \sqnorm{x_k - x_*} - 2\alpha_k \inprod{F_{\mathcal{S}_k}(\hx_k)}{x_k - \hat{x}_k} - 2\alpha_k \inprod{F_{\mathcal{S}_k}(\hx_k) - F_{\mathcal{S}_k}(x_*)}{\hx_k - x_*} \nonumber  \\
    & \quad + \alpha_k^2 \sqnorm{F_{\mathcal{S}_k}(\hx_k)} \nonumber \\
    & \le \sqnorm{x_k - x_*} - 2\alpha_k \inprod{F_{\mathcal{S}_k}(\hx_k)}{x_k - \hat{x}_k} + \alpha_k^2 \sqnorm{F_{\mathcal{S}_k}(\hx_k)} \nonumber \\
    & = \sqnorm{x_k - x_*} - \alpha_k \inprod{F_{\mathcal{S}_k}(\hx_k)}{x_k - \hat{x}_k} .\label{eqn:PolyakSEG-LS-monotone-first-bound}
\end{align}
Note that by \Cref{lemma:PolyakSEG-LS-key-properties}, we have
\begin{align}
    \alpha_k \inprod{F_{\mathcal{S}_k}(\hx_k)}{x_k - \hat{x}_k} & = \alpha_k \gamma_k \inprod{F_{\mathcal{S}_k}(\hx_k)}{F_{\mathcal{S}_k}(x_k)} \nonumber \\
    & \ge \alpha_k \gamma_k (1-A) \sqnorm{F_{\mathcal{S}_k}(x_k)} \ge \frac{(1-A) \gamma_k^2}{1+A} \sqnorm{F_{\mathcal{S}_k}(x_k)} . \label{eqn:PolyakSEG-lemma-inner-product-bound}
\end{align}
We plug this back into \eqref{eqn:PolyakSEG-LS-monotone-first-bound}, rearrange and telescope to obtain
\begin{align*}
    \frac{1}{K+1} \sum_{k=0}^K \gamma_k^2 \sqnorm{F_{\cS_k}(x_k)}
    \le \frac{(1+A) \|x_0 - x_*\|^2}{(1 - A) (K+1)} .
\end{align*}
Next, in the case where $F_\cS$ have heterogeneous strong-monotonicity parameter $\mu_\cS \ge 0$, 
using $F_{\cS_k} (x_*) = 0$ we have
\begin{align}
    \sqnorm{x_{k+1}-x_*}
    &=
    \sqnorm{
        x_k-\alpha_kF_{\mathcal S_k}(\hx_k)-x_*
    }
    \nonumber\\
    &=
    \sqnorm{x_k-x_*}
    -
    2\alpha_k
    \inprod{F_{\mathcal S_k}(\hx_k)}{x_k-x_*}
    +
    \alpha_k^2
    \sqnorm{F_{\mathcal S_k}(\hx_k)}
    \nonumber\\
    &=
    \sqnorm{x_k-x_*}
    -
    2\alpha_k
    \inprod{F_{\mathcal S_k}(\hx_k)}
            {x_k-\hx_k}
    -
    2\alpha_k
    \inprod{F_{\mathcal S_k}(\hx_k)}
            {\hx_k-x_*}
    +
    \alpha_k^2
    \sqnorm{F_{\mathcal S_k}(\hx_k)}
    \nonumber\\
    &=
    \sqnorm{x_k-x_*}
    -
    2\alpha_k
    \inprod{F_{\mathcal S_k}(\hx_k)}
            {x_k-\hx_k}
    \nonumber\\
    &\quad
    -
    2\alpha_k
    \inprod{
        F_{\mathcal S_k}(\hx_k)
        -
        F_{\mathcal S_k}(x_*)
    }{
        \hx_k-x_*
    }
    +
    \alpha_k^2
    \sqnorm{F_{\mathcal S_k}(\hx_k)}
    \nonumber\\
    &\leq
    \sqnorm{x_k-x_*}
    -
    2\alpha_k
    \inprod{F_{\mathcal S_k}(\hx_k)}
            {x_k-\hx_k}
    -
    2\alpha_k\mu_{\mathcal S_k}
    \sqnorm{\hx_k-x_*}
    +
    \alpha_k^2
    \sqnorm{F_{\mathcal S_k}(\hx_k)}
    \nonumber\\
    &=
    \sqnorm{x_k-x_*}
    -
    2\alpha_k\mu_{\mathcal S_k}
    \sqnorm{\hx_k-x_*}
    -
    \alpha_k
    \inprod{F_{\mathcal S_k}(\hx_k)}
            {x_k-\hx_k}
    \nonumber \\
    & = 
    \sqnorm{x_k - x_*} - 2\alpha_k\mu_{\mathcal S_k} \sqnorm{\hx_k-x_*} - \alpha_k\gamma_k (1-A) \sqnorm{F_{\mathcal S_k}(x_k)}
    \label{eqn:PolyakSEG-heterogeneous-first-bound}
\end{align}
where the last line uses \eqref{eqn:PolyakSEG-lemma-inner-product-bound}.
By Young's inequality, we can also bound
\begin{align}
    2\alpha_k\mu_{\mathcal S_k} \sqnorm{\hx_k-x_*} & \geq
    2\alpha_k\mu_{\mathcal S_k}
    \left(
        \frac{1-A}{1+A}
    \right)
    \sqnorm{x_k-x_*}
    -
    \alpha_k\mu_{\mathcal S_k}
    \left(
        \frac{1-A}{A}
    \right)
    \sqnorm{x_k-\hx_k}
    \nonumber\\
    &\geq
    \frac{
        2(1-A)\gamma_k\mu_{\mathcal S_k}
    }{
        (1+A)^2
    }
    \sqnorm{x_k-x_*}
    -
    \left(
        \frac{1-A}{A}
    \right)
    \alpha_k\gamma_k^2\mu_{\mathcal S_k}
    \sqnorm{F_{\mathcal S_k}(x_k)}
    \nonumber\\
    &\geq
    \frac{
        2(1-A)\gamma_k\mu_{\mathcal S_k}
    }{
        (1+A)^2
    }
    \sqnorm{x_k-x_*}
    -
    (1-A)\alpha_k\gamma_k
    \sqnorm{F_{\mathcal S_k}(x_k)} ,
    \label{eqn:PolyakSEG-heterogeneous-second-bound}
\end{align}
where the last inequality uses the $\mu_{\mathcal S_k}$-strong monotonicity of $F_{\mathcal S_k}$ and the critical condition:
\begin{align*}
    &
    \inprod{
        x_k-\hx_k
    }{
        F_{\mathcal S_k}(x_k)
        -
        F_{\mathcal S_k}(\hx_k)
    }
    \geq
    \mu_{\mathcal S_k}
    \sqnorm{x_k-\hx_k}
    \\
    &\implies
    \gamma_k\mu_{\mathcal S_k}
    \sqnorm{F_{\mathcal S_k}(x_k)}
    \leq
    \inprod{
        F_{\mathcal S_k}(x_k)
    }{
        F_{\mathcal S_k}(x_k)
        -
        F_{\mathcal S_k}(\hx_k)
    }
    \leq
    A\sqnorm{F_{\mathcal S_k}(x_k)} .
\end{align*}
The desired bound follows by plugging \eqref{eqn:PolyakSEG-heterogeneous-second-bound} into \eqref{eqn:PolyakSEG-heterogeneous-first-bound} and taking the conditional expectation.
\end{proof}

\paragraph{Implications of \Cref{theorem:PolyakSEG-convergence1}.}
If $\gamma_k$ are uniformly bounded below by some $\underline{\gamma} > 0$, then \eqref{eq:PolyakSEG-sampled-residual-rate}, together with Jensen's inequality, implies 
\begin{align*}
    \frac{1}{K+1} \sum_{k=0}^K \Exp{ \sqnorm{F(x_k)} }
    \leq
    \frac{(1+A)\|x_0-x_*\|^2}{\underline{\gamma}^2(1-A)(K+1)}.
\end{align*}
If additionally, there exists $\mu > 0$ such that $\mathbb{E}_k [\mu_{\cS_k}] \ge \mu$ for all $k \ge 0$, then \eqref{eq:PolyakSEG-heterogeneous-expectation} implies
\begin{align*}
    \Exp{\sqnorm{x_{k} - x_*}} \le \left(1 - \frac{2(1-A) \underline{\gamma} \mu }{(1+A)^2} \right)^k \sqnorm{x_0 - x_*}
\end{align*}
for any $k \ge 0$. 
Note that this second result does not require all $\mu_i$ to be positive (and consequently $F_i$ to have $x_*$ as a unique solution); we allow sample operators to be non-strongly monotone and have a larger zero set containing $x_*$, as long as they have positive curvature in expectation.
This result is qualitatively similar to the linear convergence guarantee from \cite{vaswaniPainlessStochasticGradient2019}, although they considered stochastic \algname{EG} without Polyak step-size.

\paragraph{Line-search scheme for selecting $\gamma_k$.}
We can run a line-search scheme similar to the one used in \cref{alg:PolyakEG_linesearch_unified} (with $\lambda_{-1}^1 = 0$), where we take the step-size from the previous iteration and shrink it by a fixed factor until \eqref{eqn:gamma-critical-condition-stochastic} is satisfied.
In the case where each sample operator $F_\cS$ is $L$-Lipschitz, the line-search will always terminate successfully with a finite total while loop count and the step-size will be uniformly lower-bounded over stochastic paths by the arguments similar to \cref{theorem:L0L1-linesearch-deterministic}. 
We omit the technical details to avoid repetition.

\paragraph{Necessity of interpolation.}
The preceding result assumes the interpolation condition, and this is not an artifact of the analysis.
The following example shows that this assumption cannot be removed in general: even when all component operators are Lipschitz and strongly monotone, \algname{PolyakSEG} may fail to converge even when the critical condition is satisfied at all iterations.

\begin{proposition}[Failure of \algname{PolyakSEG} without interpolation]
\label{prop:PolyakSEG-noninterpolation}
There exists a non-interpolated stochastic root-finding problem where each sample operator is Lipschitz and strongly monotone, for which \algname{PolyakSEG} fails to converge to $x_*$ on every stochastic path with constant extrapolation step-size $\gamma_k = \gamma$ satisfying the critical condition with $A \in (0,1)$ for all $k\ge 0$.
\end{proposition}

\begin{proof}
Consider $F(x) = \frac{F_1(x) + F_2(x)}{2}$, where $F_{i}\colon \reals \to \reals$ ($i=1,2$) are given by $F_{1}(x) = x+1$ and $F_{2}(x) = x-1$. 
Clearly, $F(x) = x$ has unique zero $x_* = 0$ but interpolation fails.

Suppose $\cS_{k} \in \{1,2\}$ is sampled uniformly and independently at each iteration, and write $F_{\mathcal S_{k}}(x)=x+\xi_{k}$ where $\xi_{k} \in \{+1,-1\}$.
Let $\gamma_k \equiv \gamma = A \in (0,1)$ for all $k\ge 0$.
Because $\hx_k = x_k - \gamma (x_k + \xi_{k}) = (1-\gamma) x_k - \gamma \xi_k$ and $F_{\cS_{k}}(\hx_k) = \hx_k + \xi_k = (1-\gamma) (x_k + \xi_k)$, we obtain
\begin{align*}
    \left| F_{\mathcal S_k}(\hx_k) - F_{\mathcal S_k}(x_k) \right| = \gamma |x_k + \xi_k| = A \left| F_{\cS_k} (x_k) \right| 
\end{align*}
so \eqref{eqn:gamma-critical-condition-stochastic} is satisfied.
Starting from $x_0 = 0$, \algname{PolyakSEG} updates are always well-defined and $|x_k| < 1$ for $k\ge 0$.
Indeed, if $|x_k| < 1$ then $x_k + \xi_k \ne 0$, so $F_{\cS_k} (\hx_k) \ne 0$ and $\alpha_k = \frac{\gamma}{1-\gamma}$. 
This then implies $x_{k+1} = x_k - \alpha_k F_{\cS_k}(\hx_k) = (1-\gamma) x_k - \gamma \xi_k$ and $|x_{k+1}| < 1$ by induction hypothesis, completing the induction.
Finally, because $\xi_k \in \{+1, -1\}$, we have $\gamma = |(1-\gamma) x_k - x_{k+1}| \le (1-\gamma) |x_k| + |x_{k+1}|$ for all $k\ge 0$, so the sequence $x_k$ cannot converge to $x_* = 0$.
\end{proof}

The above example shows that interpolation cannot be simply removed from the convergence guarantee for \algname{PolyakSEG}, and another mechanism is needed for convergence.
The next section analyzes \algname{DecPolyakSEG} (\cref{alg:decPolyakSEG}), an algorithmic modification that restores convergence without interpolation.

\subsection{\algname{DecPolyakSEG}: Residual Convergence without Interpolation}
\label{subsec:DecPolyakSEGNEW}

Here, we propose and analyze \algname{DecPolyakSEG} (\Cref{alg:decPolyakSEG}), whose convergence analysis does not require the interpolation condition used for \algname{PolyakSEG}. 
The key distinction from \algname{PolyakSEG} is that we let both step-size sequences decrease, as highlighted in green in \Cref{alg:decPolyakSEG}. 
Namely, we enforce
\[
    \gamma_k \le \frac{c_{k-1}}{c_k}\gamma_{k-1}
    \qquad
    \text{and}
    \qquad
    \alpha_k\le \alpha_{k-1}.
\]
The first condition is analogous to the decreasing step-size mechanism in \algname{DecSPS} of \citet{orvietoDynamicsSGDStochastic2022}, whereas the second condition is specific to our adaptive Polyak-type extragradient update. 
Similar to \algname{DecSPS}, our analysis of \algname{DecPolyakSEG} requires the following localization assumption on the trajectory, i.e., the iterates do not escape to infinity during its run.

\begin{assumption}
\label{assumption:bounded-iterates}
There exists a compact convex set $\mathcal{C} \subseteq \reals^d$ with diameter $D$, containing a zero $x_*$ of $F$ such that
$
    x_k,\hx_k\in\mathcal C
$
during the runtime of \algname{DecPolyakSEG}.
\end{assumption}

This condition can be derived if the problem has additional favorable structure such as strong monotonicity or interpolation, or if the algorithm has a stabilizing mechanism such as projection steps (with constrained problems).
To keep the exposition coherent with the previous sections, here we focus on the unconstrained monotone case and the algorithm without projection, and present our convergence analysis under Assumption~\ref{assumption:bounded-iterates}.

\subsubsection{Residual convergence for monotone problems} 

We provide the convergence theorem in terms of the expected squared residual norm $\Exp{\|F(\overline{x}_K)\|^2}$ for the averaged iterate $\overline{x}_K$,
where we use the same choice $c_k=\sqrt{k+1}$ as in the decreasing-step construction used by \citet{orvietoDynamicsSGDStochastic2022} in the analysis of \algname{DecSPS}.

\begin{theorem} 
\label{theorem:DecPolyakSEGLS_monotone}
Let each $F_i\colon \R^d \to \R^d$ be $L$-Lipschitz, and let $F$ be monotone.
Let $x_* \in \mathcal{C}$ be a zero of $F$, i.e., $F(x_*) = 0$.
Let \algname{DecPolyakSEG} run with $\gamma_{-1} > 0$, $c_{-1} = 1$, $c_k = \sqrt{k+1}$ for $k\ge 0$,
and $\gamma_k$ satisfying $\gamma_k \leq \frac{c_{k-1}}{c_k}\gamma_{k-1}$ and the stochastic critical condition \eqref{eqn:gamma-critical-condition-stochastic} for $k\ge 0$, with $A \in (0,1)$.
If Assumption~\ref{assumption:bounded-iterates} holds, then $\overline{x}_K := \frac{1}{K+1}\sum_{k=0}^K\hx_k$ satisfies 
\begin{equation}
\label{eq:DecPolyakSEG-direct-residual-bound}
\begin{aligned}
    \Exp{\|F(\overline{x}_K)\|^2}
    \leq
    4L\Bigg[
        &\frac{1}{K+1}
        \Exp{\frac{9D^2}{8\alpha_K}}
        +
        \frac{3\sigma D}{2\sqrt{K+1}}
        +
        \frac{2\gamma_{-1}}
        {(1-A)\sqrt{K+1}}
        \left(
            \sigma^2 + 9L^2D^2
        \right)
    \Bigg]
\end{aligned}
\end{equation}
where $\sigma^2 = 2L^2D^2 + 2 \mathbb{E}_\cS \left[ \|F_{\mathcal S}(x_*)\|^2 \right]$.
\end{theorem}

\begin{proof}
Note that since each component operator is $L$-Lipschitz, both $F_{\mathcal S_k}$ and $F$ are $L$-Lipschitz.
Fix any $u\in\R^d$. From the update rule $x_{k+1} = x_k-\alpha_kF_{\mathcal S_k}(\hx_k)$, we obtain
\begin{align}
    \|x_{k+1}-u\|^2
    &=
    \left\|
        x_k-u-\alpha_kF_{\mathcal S_k}(\hx_k)
    \right\|^2
    \nonumber\\
    &=
    \|x_k-u\|^2
    -
    2\alpha_k
    \inprod{F_{\mathcal S_k}(\hx_k)}{x_k-u}
    +
    \alpha_k^2
    \|F_{\mathcal S_k}(\hx_k)\|^2.
    \label{eq:DecPolyakSEG-residual-distance-expansion}
\end{align}
Rearranging \eqref{eq:DecPolyakSEG-residual-distance-expansion} gives

\begin{align}
    \inprod{F_{\mathcal S_k}(\hx_k)}{\hx_k-u}
    &=
    \inprod{F_{\mathcal S_k}(\hx_k)}{x_k-u}
    -
    \inprod{F_{\mathcal S_k}(\hx_k)}{x_k-\hx_k}
    \nonumber\\
    &=
    \frac{
        \|x_k-u\|^2-\|x_{k+1}-u\|^2
    }{
        2\alpha_k
    }
    +
    \frac{\alpha_k}{2}
    \|F_{\mathcal S_k}(\hx_k)\|^2
    -
    \inprod{F_{\mathcal S_k}(\hx_k)}{x_k-\hx_k} \nonumber \\
    &\leq
    \frac{
        \|x_k-u\|^2-\|x_{k+1}-u\|^2
    }{
        2\alpha_k
    }
    -
    \frac12
    \inprod{F_{\mathcal S_k}(\hx_k)}{x_k-\hx_k}
    \nonumber\\
    &=
    \frac{
        \|x_k-u\|^2-\|x_{k+1}-u\|^2
    }{
        2\alpha_k
    }
    -
    \frac{\gamma_k}{2}
    \inprod{
        F_{\mathcal S_k}(\hx_k)
    }{
        F_{\mathcal S_k}(x_k)
    } 
    \label{eq:DecPolyakSEG-residual-sampled-comparison}
\end{align}
where the third line uses $\alpha_k \leq \frac{ \inprod{F_{\mathcal S_k}(\hx_k)}{x_k-\hx_k} }{ \|F_{\mathcal S_k (\hx_k)\|^2 }}$.
Now combining the identity
\begin{align*}
    \inprod{F(\hx_k)}{\hx_k-u}
    =
    \inprod{F_{\mathcal S_k}(\hx_k)}{\hx_k-u}
    +
    \inprod{
        F_{\mathcal S_k}(\hx_k)-F(\hx_k)
    }{
        u-\hx_k
    }
\end{align*}
with \eqref{eq:DecPolyakSEG-residual-sampled-comparison} yields
\begin{align}
    \inprod{F(\hx_k)}{\hx_k-u}
    &\leq
    \frac{
        \|x_k-u\|^2-\|x_{k+1}-u\|^2
    }{
        2\alpha_k
    }
    -
    \frac{\gamma_k}{2}
    \inprod{
        F_{\mathcal S_k}(\hx_k)
    }{
        F_{\mathcal S_k}(x_k)
    }
    \nonumber\\
    &\quad
    +
    \inprod{
        F_{\mathcal S_k}(\hx_k)-F(\hx_k)
    }{
        u-\hx_k
    } .
    \label{eq:DecPolyakSEG-residual-comparison-start}
\end{align}
We control the last term in \eqref{eq:DecPolyakSEG-residual-comparison-start} as
\begin{align}
    & \inprod{F_{\mathcal S_k}(\hx_k)-F(\hx_k)}{u-\hx_k} \nonumber \\
    & =
    \inprod{F_{\mathcal S_k}(\hx_k)-F_{\mathcal S_k}(x_k)}{u-\hx_k}
    + \inprod{F(x_k)-F(\hx_k)}{u-\hx_k}
    + \inprod{F_{\mathcal S_k}(x_k)-F(x_k)}{u-\hx_k} \nonumber \\
    & \le \left\|
        F_{\mathcal S_k}(\hx_k)-F_{\mathcal S_k}(x_k)
    \right\|
    \|u-\hx_k\|
    +
    \left\|
        F(x_k)-F(\hx_k)
    \right\|
    \|u-\hx_k\| 
    + \inprod{F_{\mathcal S_k}(x_k)-F(x_k)}{u-\hx_k} \nonumber \\
    & \le
    2L\|x_k-\hx_k\|\|u-\hx_k\| + \inprod{F_{\mathcal S_k}(x_k)-F(x_k)}{u-\hx_k} .
    \label{eq:DecPolyakSEG-noise-exact-decomposition}
\end{align}
where the last inequality uses $L$-Lipschitzness of $F_{\mathcal S_k}$ and $F$.
Finally, for the last term, 
\begin{align}
    &
    \inprod{
        F_{\mathcal S_k}(x_k)-F(x_k)
    }{
        u-\hx_k
    }
    \nonumber\\
    &=
    \inprod{
        F_{\mathcal S_k}(x_k)-F(x_k)
    }{
        u-x_0
    }
    +
    \inprod{
        F_{\mathcal S_k}(x_k)-F(x_k)
    }{
        x_0-x_k
    }
    +
    \inprod{
        F_{\mathcal S_k}(x_k)-F(x_k)
    }{
        x_k-\hx_k
    }
    \nonumber \\
    & \le
    \inprod{
        F_{\mathcal S_k}(x_k)-F(x_k)
    }{
        u-x_0
    }
    +
    \inprod{
        F_{\mathcal S_k}(x_k)-F(x_k)
    }{
        x_0-x_k
    }
    +
    \|F_{\mathcal S_k}(x_k)-F(x_k)\|
    \|x_k-\hx_k\| .
    \label{eq:DecPolyakSEG-noise-vector-decomposition}
\end{align}
Combining \eqref{eq:DecPolyakSEG-noise-exact-decomposition} and \eqref{eq:DecPolyakSEG-noise-vector-decomposition} and using $x_k-\hx_k = \gamma_kF_{\mathcal S_k}(x_k)$ we obtain
\begin{align}
    \inprod{
        F_{\mathcal S_k}(\hx_k)-F(\hx_k)
    }{
        u-\hx_k
    }
    & \leq
    \inprod{
        F_{\mathcal S_k}(x_k)-F(x_k)
    }{
        u-x_0
    }
    +
    \inprod{
        F_{\mathcal S_k}(x_k)-F(x_k)
    }{
        x_0-x_k
    }
    \nonumber\\
    & \quad
    +
    2L\gamma_k
    \|F_{\mathcal S_k}(x_k)\|
    \|u-\hx_k\|
    +
    \gamma_k
    \|F_{\mathcal S_k}(x_k)-F(x_k)\|
    \|F_{\mathcal S_k}(x_k)\|.
    \label{eq:DecPolyakSEG-residual-noise-decomposition}
\end{align}
We plug \eqref{eq:DecPolyakSEG-residual-noise-decomposition} back into
\eqref{eq:DecPolyakSEG-residual-comparison-start} and apply \cref{lemma:PolyakSEG-LS-key-properties}(b): $\inprod{F_{\mathcal S_k}(\hx_k)}{F_{\mathcal S_k}(x_k)} \geq (1-A)\|F_{\mathcal S_k}(x_k)\|^2$, and use Young's inequality to bound the following terms:
\begin{align*}
    2L
    \|F_{\mathcal S_k}(x_k)\|
    \|u-\hx_k\|
    &\leq
    \frac{1-A}{4}
    \|F_{\mathcal S_k}(x_k)\|^2
    +
    \frac{4L^2}{1-A}\|u-\hx_k\|^2 \\
    \|F_{\mathcal S_k}(x_k)-F(x_k)\|
    \|F_{\mathcal S_k}(x_k)\|
    &\leq
    \frac{1-A}{4}
    \|F_{\mathcal S_k}(x_k)\|^2
    +
    \frac{1}{1-A}
    \|F_{\mathcal S_k}(x_k)-F(x_k)\|^2.
\end{align*}
The resulting bound is 
\begin{align}
    \inprod{F(\hx_k)}{\hx_k-u} 
    & \leq
    \frac{
        \|x_k-u\|^2-\|x_{k+1}-u\|^2
    }{
        2\alpha_k
    }
    +
    \inprod{
        F_{\mathcal S_k}(x_k)-F(x_k)
    }{
        u-x_0
    }
    \nonumber \\
    &\quad
    +
    \inprod{
        F_{\mathcal S_k}(x_k)-F(x_k)
    }{
        x_0-x_k
    }
    +
    \frac{\gamma_k}{1-A}
    \left(
        \|F_{\mathcal S_k}(x_k)-F(x_k)\|^2
        +
        4L^2\|u-\hx_k\|^2
    \right).
    \label{eq:DecPolyakSEG-residual-comparison}
\end{align}
The above holds pathwisely and for any $u \in \reals^d$; hence, we can sum this up for $k=0,\dots,K$ and substitute $u = u_K := \overline{x}_K - \frac{1}{2L}F(\overline{x}_K)$.
By Assumption~\ref{assumption:bounded-iterates}, we have $\overline{x}_K \in \cC$ (as it is a convex combination of $\hx_k \in \cC$), and therefore $\norm{\overline{x}_K - x_*} \le D$ and $\norm{F(\overline{x}_K)} \le L\norm{\overline{x}_K - x_*} \le LD$.
This implies that for every $z\in \cC$,
\begin{align}
\label{eq:DecPolyakSEG-uK-distance}
    \norm{u_K - z} \le \norm{u_K - \overline{x}_K} + \norm{\overline{x}_K - z} = \frac{\norm{F(\overline{x}_K)}}{2L} + D \le \frac{3D}{2} := R .
\end{align}
Using \eqref{eq:DecPolyakSEG-uK-distance} and nonincreasingness of $\alpha_k$, we obtain
\begin{align}
    &
    \sum_{k=0}^K
    \frac{
        \|x_k-u_K\|^2-\|x_{k+1}-u_K\|^2
    }{
        2\alpha_k
    } \nonumber \\
    & =
    \frac{\|x_0-u_K\|^2}{2\alpha_0}
    -
    \frac{\|x_{K+1}-u_K\|^2}{2\alpha_K}
    +
    \sum_{k=1}^K
    \left(
        \frac{1}{2\alpha_k}
        -
        \frac{1}{2\alpha_{k-1}}
    \right)
    \|x_k-u_K\|^2 \nonumber \\
    & \leq
    \frac{R^2}{2\alpha_0}
    +
    R^2
    \sum_{k=1}^K
    \left(
        \frac{1}{2\alpha_k}
        -
        \frac{1}{2\alpha_{k-1}}
    \right) =
    \frac{R^2}{2\alpha_K}.
    \label{eq:DecPolyakSEG-weighted-telescoping-bound}
\end{align}
Summing \eqref{eq:DecPolyakSEG-residual-comparison} for $k=0,\dots,K$ with $u=u_K$ and using \eqref{eq:DecPolyakSEG-weighted-telescoping-bound} and $\|u_K-\hx_k\|\leq R$ yields
\begin{align}
    & \sum_{k=0}^K
    \inprod{F(\hx_k)}{\hx_k-u_K} \nonumber \\
    & \leq
    \frac{R^2}{2\alpha_K}
    +
    \inprod{
        \sum_{k=0}^K
        \bigl(
            F_{\mathcal S_k}(x_k)-F(x_k)
        \bigr)
    }{
        u_K-x_0
    }
    \nonumber\\
    &\quad
    +
    \sum_{k=0}^K
    \inprod{
        F_{\mathcal S_k}(x_k)-F(x_k)
    }{
        x_0-x_k
    }
    +
    \frac{1}{1-A}
    \sum_{k=0}^K
    \gamma_k
    \Bigl(
        \|F_{\mathcal S_k}(x_k)-F(x_k)\|^2
        +
        4L^2R^2
    \Bigr).
    \label{eq:DecPolyakSEG-residual-summed}
\end{align}

We next bound the right hand side of \eqref{eq:DecPolyakSEG-residual-summed} in expectation.
Again, denote by $\Expk{\cdot}$ the conditional expectation with respect to the randomness revealed before drawing $\mathcal S_k$.
Then
\begin{align*}
    \Expk{
        \|F_{\mathcal S_k}(x_k)-F(x_k)\|^2
    }
    & =
    \Expk{\|F_{\mathcal S_k}(x_k)\|^2}
    -
    \|F(x_k)\|^2
    \\
    & \leq
    2\Expk{
        \|F_{\mathcal S_k}(x_k)
          -F_{\mathcal S_k}(x_*)\|^2
    }
    +
    2\Expk{\|F_{\mathcal S_k}(x_*)\|^2} \nonumber\\
    & \leq
    2L^2\|x_k-x_*\|^2
    +
    2\Exp{\|F_{\mathcal S}(x_*)\|^2}
    \leq
    \sigma^2
\end{align*}
so $\Exp{\sqnorm{F_{\mathcal S_k}(x_k)-F(x_k)}} \le \sigma^2$ by the tower property.
Next, because $F_{\cS_j}(x_j) - F(x_j)$ for $j<k$ is measurable with respect to information revealed before drawing $\cS_k$ and $\Expk{F_{\cS_k}(x_k) - F(x_k)} = 0$, we have $\Exp{
    \inprod{
        F_{\mathcal S_j}(x_j)-F(x_j)
    }{
        F_{\mathcal S_k}(x_k)-F(x_k)
    }
}
= 0$ and therefore, 
\begin{align*}
    & \Exp{
        \left\|
            \sum_{k=0}^K
            \bigl(
                F_{\mathcal S_k}(x_k)-F(x_k)
            \bigr)
        \right\|^2
    } \\
    & =
    \sum_{k=0}^K
    \Exp{
        \|F_{\mathcal S_k}(x_k)-F(x_k)\|^2
    }
    +
    2
    \sum_{0\leq j<k\leq K}
    \Exp{
        \inprod{
            F_{\mathcal S_j}(x_j)-F(x_j)
        }{
            F_{\mathcal S_k}(x_k)-F(x_k)
        }
    }
    \\
    &=
    \sum_{k=0}^K
    \Exp{
        \|F_{\mathcal S_k}(x_k)-F(x_k)\|^2
    }
    \\
    &\leq
    (K+1)\sigma^2 .
\end{align*}
Using $\norm{u_k - x_0} \le R$ and Cauchy--Schwarz inequality, we then obtain
\begin{align}
    \Exp{
        \inprod{
            \sum_{k=0}^K
            \bigl(
                F_{\mathcal S_k}(x_k)-F(x_k)
            \bigr)
        }{
            u_K-x_0
        }
    }
    & \leq
    \Exp{
        \left\|
            \sum_{k=0}^K
            \bigl(
                F_{\mathcal S_k}(x_k)-F(x_k)
            \bigr)
        \right\|
        \|u_K-x_0\|
    }
    \nonumber\\
    &\leq
    R
    \Exp{
        \left\|
            \sum_{k=0}^K
            \bigl(
                F_{\mathcal S_k}(x_k)-F(x_k)
            \bigr)
        \right\|
    }
    \nonumber\\
    &\leq
    R
    \left(
        \Exp{
            \left\|
                \sum_{k=0}^K
                \bigl(
                    F_{\mathcal S_k}(x_k)-F(x_k)
                \bigr)
            \right\|^2
        }
    \right)^{1/2}
    \nonumber\\
    &\leq
    \sigma R\sqrt{K+1}.
    \label{eq:DecPolyakSEG-martingale-sum}
\end{align}
Furthermore, using $\Expk{F_{\cS_k}(x_k) - F(x_k)} = 0$ again and applying the tower rule, we have
\begin{align}
    \Exp{
        \sum_{k=0}^K
        \inprod{
            F_{\mathcal S_k}(x_k)-F(x_k)
        }{
            x_0-x_k
        }
    }
    =
    0 .
\label{eq:DecPolyakSEG-centered-cross-term}
\end{align}
Finally, the step-size condition for \algname{DecPolyakSEG} gives $c_k\gamma_k \le c_{k-1}\gamma_{k-1} \le \cdots \le c_{-1}\gamma_{-1} = \gamma_{-1}$, 
and since we take $c_k = \sqrt{k+1}$, we have $\gamma_k \le \frac{\gamma_{-1}}{\sqrt{k+1}}$.
Hence
\begin{align}
    &
    \Exp{
        \sum_{k=0}^K
        \gamma_k
        \Bigl(
            \|F_{\mathcal S_k}(x_k)-F(x_k)\|^2
            +
            4L^2R^2
        \Bigr)
    }
    \leq
    \sum_{k=0}^K
    \frac{\gamma_{-1}}{\sqrt{k+1}}
    \Exp{
        \|F_{\mathcal S_k}(x_k)-F(x_k)\|^2
        +
        4L^2R^2
    }
    \nonumber\\
    &\leq
    \gamma_{-1}
    \left(
        \sigma^2+4L^2R^2
    \right)
    \sum_{k=0}^K\frac{1}{\sqrt{k+1}}
    \leq
    2\gamma_{-1}
    \left(
        \sigma^2+4L^2R^2
    \right)
    \sqrt{K+1}.
    \label{eq:DecPolyakSEG-weighted-noise-bound}
\end{align}
Combining \eqref{eq:DecPolyakSEG-residual-summed}, \eqref{eq:DecPolyakSEG-martingale-sum}, \eqref{eq:DecPolyakSEG-centered-cross-term} and \eqref{eq:DecPolyakSEG-weighted-noise-bound} gives
\begin{align}
    \Exp{
        \sum_{k=0}^K
        \inprod{F(\hx_k)}{\hx_k-u_K}
    }
    \leq
    \Exp{\frac{R^2}{2\alpha_K}}
    +
    \sigma R\sqrt{K+1}
    +
    \frac{2\gamma_{-1}}{1-A}
    \left(
        \sigma^2+4L^2R^2
    \right)
    \sqrt{K+1}.
    \label{eq:DecPolyakSEG-residual-upper}
\end{align}
By monotonicity of $F$, we have $\inprod{F(\hx_k)}{\hx_k-u_K} \geq \inprod{F(u_K)}{\hx_k-u_K}$.
Using this to lower bound the left hand side of \eqref{eq:DecPolyakSEG-residual-upper} and summing over $k=0,\dots,K$ we obtain
\begin{align*}
    \Exp{
        \sum_{k=0}^K
        \inprod{F(\hx_k)}{\hx_k-u_K}
    } \ge
    \Exp{
        \sum_{k=0}^K
        \inprod{F(u_K)}{\hx_k-u_K}
    } 
    = (K+1) \Exp{\inprod{F(u_K)}{\overline{x}_K - u_K}} .
\end{align*}
Because $\overline{x}_K - u_K = \frac{1}{2L} F(\overline{x}_K)$, by $L$-Lipschitzness of $F$ gives
\begin{align*}
    \Exp{\inprod{F(u_K)}{\overline{x}_K - u_K}} & = \frac{1}{2L} \Exp{\inprod{F(u_K)}{F(\overline{x}_K)}} \\
    & = \frac{1}{2L} \Exp{\|F(\overline{x}_K)\|^2
        +
        \inprod{
            F(u_K)-F(\overline{x}_K)
        }{
            F(\overline{x}_K)
        }} \\
    & \geq
    \frac{1}{2L} \Exp{
        \|F(\overline{x}_K)\|^2
        -
        L\|u_K-\overline{x}_K\|
        \|F(\overline{x}_K)\|
    } =
    \frac{1}{4L} \Exp{\sqnorm{F(\overline{x}_K)}} .
\end{align*}
Combining this with \eqref{eq:DecPolyakSEG-residual-upper}, dividing by $\frac{K+1}{4L}$ and plugging in $R = \frac{3D}{2}$ proves the desired bound.

\end{proof}

\subsubsection{Step-size choices and line-search}
\Cref{theorem:DecPolyakSEGLS_monotone} does not require a particular procedure for selecting the extrapolation step-size $\gamma_k$. 
It applies to any sequence satisfying $\gamma_k \leq \frac{c_{k-1}}{c_k}\gamma_{k-1}$ and the stochastic critical condition \eqref{eqn:gamma-critical-condition-stochastic}.
For example, when the Lipschitz constant $L$ is known, a simple admissible choice is $\gamma_k = \frac{A}{L\sqrt{k+1}}$.
When $L$ is unknown, the critical condition can instead be enforced by line-search. 
We provide a basic and effective implementation in \cref{alg:decPolyakSEG_linesearch1}, which starts from the largest permitted choice $\gamma_k = \frac{c_{k-1}}{c_k} \gamma_{k-1}$ and geometrically decreases it \eqref{eqn:gamma-critical-condition-stochastic} until is satisfied.
We provide the corollary of \cref{theorem:DecPolyakSEGLS_monotone}, which captures both cases mentioned above.

\begin{algorithm}[ht]
\small
    \caption{\algname{\textcolor{PineGreen}{DecPolyakSEG-LS}}}
    \label{alg:decPolyakSEG_linesearch1}
    \begin{algorithmic}[1]
        \REQUIRE Initial point $x_0 \in \R^d$, initial step-size $\gamma_{-1} > 0$, $\alpha_{-1} = \infty$, line-search factor $\beta > 1$, $A \in (0, 1)$ and a non-decreasing sequence $\{ c_k\}_{k=-1}^{\infty}$.
        \FOR{$k = 0, 1,...,K$}
        \STATE $\gamma_k = \frac{c_{k-1}}{c_{k}}\gamma_{k-1}$.
        \STATE Sample $\mathcal{S}_k \subseteq [n]$.
        \STATE $\hx_k = x_k - \gamma_k F_{\mathcal{S}_k}(x_k)$.
        \WHILE{$\| F_{\mathcal{S}_k}(x_k) - F_{\mathcal{S}_k}(\hx_k) \| > A \|F_{\mathcal{S}_k}(x_k)\|$}
            \STATE $\gamma_k = \gamma_k / \beta$
            \STATE $\hx_k = x_k - \gamma_k F_{\mathcal{S}_k}(x_k)$
        \ENDWHILE
        \vspace{2mm}
        \STATE $\alpha_k = \min \left\{ \frac{\la F_{\mathcal{S}_k}(\hx_k), x_k - \hx_k \ra}{\|F_{\mathcal{S}_k}(\hx_k)\|^2}, \alpha_{k-1}\right\}$.
        \vspace{2mm}
        \STATE $x_{k+1} = x_k - \alpha_k F_{\mathcal{S}_k}(\hx_k)$.
        \ENDFOR
    \end{algorithmic}
\end{algorithm}

\begin{corollary} 
\label{corollary:DecPolyakSEG-explicit-residual-rate}
Under the assumptions of \cref{theorem:DecPolyakSEGLS_monotone}, suppose 
\begin{equation}
\label{eq:DecPolyakSEG-lower-envelope}
    \gamma_k
    \geq
    \frac{\underline{\gamma}}{\sqrt{k+1}}
    \qquad
    \forall k = 0,1,\dots 
\end{equation}
almost surely for some $\underline{\gamma}>0$.
Then
\begin{equation}
\label{eq:DecPolyakSEG-direct-residual-rate}
\begin{aligned}
    \Exp{\|F(\overline{x}_K)\|^2}
    \leq
    \frac{4L}{\sqrt{K+1}}
    \Bigg[
        &\frac{9(1+A)D^2}{8\underline{\gamma}}
        +
        \frac{3\sigma D}{2} 
        +
        \frac{2\gamma_{-1}}{1-A}
        \left(
            \sigma^2 + 9L^2D^2
        \right)
    \Bigg]
    =
    \mathcal O\left(\frac{1}{\sqrt{K+1}}\right) .
\end{aligned}
\end{equation}
The condition~\eqref{eq:DecPolyakSEG-lower-envelope} holds in both of the following cases:
\begin{enumerate}
    \item For $\gamma_{-1} = \frac{A}{L}$ and $\gamma_k = \frac{A}{L\sqrt{k+1}}$ for $k\ge 0$, we have \eqref{eq:DecPolyakSEG-lower-envelope} with $\underline{\gamma} = \frac{A}{L}$.
    
    \item For \algname{DecPolyakSEG-LS} (\cref{alg:decPolyakSEG_linesearch1}) with $c_k=\sqrt{k+1}$, we have \eqref{eq:DecPolyakSEG-lower-envelope} with 
    \[
        \underline{\gamma}
        =
        \min\left\{
            \gamma_{-1},
            \frac{A}{\beta L}
        \right\} .
    \]    
\end{enumerate}
\end{corollary}

\begin{proof}
Assuming \eqref{eq:DecPolyakSEG-lower-envelope}, by \Cref{lemma:PolyakSEG-LS-key-properties}, we have $\alpha_K \geq \frac{\gamma_K}{1+A} \geq \frac{\underline{\gamma}}{(1+A)\sqrt{K+1}}$, and thus
\[
    \frac{1}{K+1}
    \Exp{\frac{R^2}{\alpha_K}}
    \leq
    \frac{(1+A)R^2}
         {\underline{\gamma}\sqrt{K+1}}.
\]
Substituting this into \eqref{eq:DecPolyakSEG-direct-residual-bound} immediately yields \eqref{eq:DecPolyakSEG-direct-residual-rate}. 

Since Case~1 is evident, we consider Case~2 on \algname{DecPolyakSEG-LS}.
Observe that any $\gamma_k \le \frac{A}{L}$ satisfies the critical condition~\eqref{eqn:gamma-critical-condition-stochastic} by $L$-Lipschitzness of $F$. 
Thus, if $\gamma_{-1} \le \frac{A}{L}$ then we will have $\gamma_k = \frac{c_{-1}}{c_k} \gamma_{-1} = \frac{\gamma_{-1}}{\sqrt{k+1}}$.
Hence, we may assume $\gamma_{-1} > \frac{A}{L}$. 
Now at any iteration $k \ge 0$, let $\widetilde{\gamma}_k = \frac{c_{k-1}}{c_k}\gamma_{k-1}$ be the trial step-size where line-search starts from.
If $\widetilde{\gamma}_k\leq \frac{A}{L}$, it is accepted without backtracking.
Otherwise, the accepted step-size will be $\gamma_k = \beta^{-m_k} \widetilde{\gamma}_k$ for some $m_k \ge 1$, where $\beta^{-m_k + 1} \widetilde{\gamma}_k > \frac{A}{L}$,
which implies $\gamma_k > \frac{A}{\beta L} \ge \frac{A}{\sqrt{k+1} \beta L}$.
Together with the case $\gamma_{-1} \le \frac{A}{L}$, this proves that \eqref{eq:DecPolyakSEG-lower-envelope} holds with $\underline{\gamma} = \min\left\{ \gamma_{-1}, \frac{A}{\beta L} \right\}$.
\end{proof}

The previous results rely on the localization condition (Assumption~\ref{assumption:bounded-iterates}) to guarantee convergence as it is commonly done in analyses of stochastic and adaptive algorithms.
This type of analysis isolates the stability requirement on the algorithm trajectory that does not necessarily hold for general monotone stochastic problems.
Nevertheless, the localization assumption can be removed when the problem has an additional structure such as samplewise strong monotonicity, as the following proposition shows.

\begin{proposition}
\label{proposition:DecPolyakSEG_strong_monotone}
Let each $F_i$ be $L$-Lipschitz and $\mu$-strongly monotone, and let
$x_*$ satisfy $F(x_*)=0$. 
Let $\{x_k\}$ be generated by \algname{DecPolyakSEG} with $\gamma_k$ satisfying the stochastic critical condition \eqref{eqn:gamma-critical-condition-stochastic} for $k\ge 0$ with $A \in (0,1)$.
Then, almost surely, 
\begin{align}
    \sup_{k \ge 0} \sqnorm{x_k-x_*}
    \leq
    \max\left\{
        \sqnorm{x_0-x_*},
        \frac{1}{\mu^2(1-A)}
        \underset{1\le i \le n}{\max} \norm{F_i(x_*)}
    \right\} < \infty .
\label{eq:DecPolyakSEG-strong-monotone-stability}
\end{align}
Consequently, the extrapolated points $\hx_k$
also stay bounded almost surely.
\end{proposition}

\begin{proof}
First, note that since each $F_i$ is $L$-Lipschitz and $\mu$-strongly monotone,
every mini-batch operator $F_{\mathcal S_k}$ is also $L$-Lipschitz and
$\mu$-strongly monotone. 
By the choice of $\alpha_k$ in \algname{DecPolyakSEG}, we have $\alpha_k \sqnorm{F_{\mathcal S_k}(\hx_k)} \leq \inprod{F_{\mathcal S_k}(\hx_k)}{x_k-\hx_k}$, so
\begin{align}
    \sqnorm{x_{k+1}-x_*} 
    &=
    \sqnorm{x_k-x_*}
    -2\alpha_k
    \inprod{F_{\mathcal S_k}(\hx_k)}{x_k-x_*}
    +\alpha_k^2\sqnorm{F_{\mathcal S_k}(\hx_k)}
    \nonumber \\
    & \leq
    \sqnorm{x_k-x_*}
    -\alpha_k
    \inprod{F_{\mathcal S_k}(\hx_k)}{x_k-\hx_k}
    -2\alpha_k
    \inprod{F_{\mathcal S_k}(\hx_k)}{\hx_k-x_*} .
    \label{eqn:strong-monotone-localization-eq1}
\end{align}
Strong monotonicity of $F_{\mathcal S_k}$ and Young's inequality yield
\begin{align*}
    -2\inprod{F_{\mathcal S_k}(\hx_k)}{\hx_k-x_*}
    &=
    -2\inprod{
        F_{\mathcal S_k}(\hx_k)-F_{\mathcal S_k}(x_*)
    }{
        \hx_k-x_*
    }
    -2\inprod{F_{\mathcal S_k}(x_*)}{\hx_k-x_*}
    \\
    &\leq
    -2\mu\sqnorm{\hx_k-x_*}
    +\frac{1}{\mu}\sqnorm{F_{\mathcal S_k}(x_*)}
    +\mu\sqnorm{\hx_k-x_*}
    \\
    &=
    -\mu\sqnorm{\hx_k-x_*}
    +\frac{1}{\mu}\sqnorm{F_{\mathcal S_k}(x_*)}.
\end{align*}
Applying this to \eqref{eqn:strong-monotone-localization-eq1} and using $x_k-\hx_k=\gamma_kF_{\mathcal S_k}(x_k)$, it follows that
\begin{align}
    \sqnorm{x_{k+1}-x_*}
    &\leq
    \sqnorm{x_k-x_*}
    -\alpha_k\mu\sqnorm{\hx_k-x_*}
    +\frac{\alpha_k}{\mu}
    \sqnorm{F_{\mathcal S_k}(x_*)}
    -\alpha_k\gamma_k
    \inprod{
        F_{\mathcal S_k}(\hx_k)
    }{
        F_{\mathcal S_k}(x_k)
    } .
\label{eq:DecPolyakSEG-strong-monotone-first-step}
\end{align}
Note that by Young's inequality
\begin{align*}
    -\sqnorm{\hx_k-x_*}
    \leq
    -(1-A)\sqnorm{x_k-x_*}
    +\frac{1-A}{A}\sqnorm{x_k-\hx_k} 
\end{align*}
holds, and strong monotonicity of $F_{\mathcal S_k}$ gives
\begin{align*}
    \mu\sqnorm{x_k-\hx_k}
    &\leq
    \inprod{
        F_{\mathcal S_k}(x_k)-F_{\mathcal S_k}(\hx_k)
    }{
        x_k-\hx_k
    }
    =
    \gamma_k
    \inprod{
        F_{\mathcal S_k}(x_k)
    }{
        F_{\mathcal S_k}(x_k)-F_{\mathcal S_k}(\hx_k)
    }
\end{align*}
Substituting these two inequalities into
\eqref{eq:DecPolyakSEG-strong-monotone-first-step} yields
\begin{align}
    \sqnorm{x_{k+1}-x_*}
    &\leq
    \left(1-(1-A)\alpha_k\mu\right)
    \sqnorm{x_k-x_*}
    \nonumber \\
    &\quad
    +\frac{\alpha_k\gamma_k}{A}
    \left(
        (1-A)\sqnorm{F_{\mathcal S_k}(x_k)}
        -
        \inprod{
            F_{\mathcal S_k}(x_k)
        }{
            F_{\mathcal S_k}(\hx_k)
        }
    \right)
    +\frac{\alpha_k}{\mu}
    \sqnorm{F_{\mathcal S_k}(x_*)}
    \nonumber \\
    & \le \left(1-(1-A)\alpha_k\mu\right)
    \sqnorm{x_k-x_*}
    +
    \frac{\alpha_k}{\mu}
    \sqnorm{F_{\mathcal S_k}(x_*)}  
    \nonumber \\
    & \le 
    \left(1-(1-A)\alpha_k\mu\right)
    \sqnorm{x_k-x_*}
    +
    (1-A) \alpha_k \mu \frac{\sqnorm{F_{\mathcal S_k}(x_*)}}{\mu^2 (1-A)} 
    \label{eq:DecPolyakSEG-strong-monotone-recursion}
\end{align}
where the second last inequality follows from \cref{lemma:PolyakSEG-LS-key-properties}(b).
Applying this recursively, we obtain \eqref{eq:DecPolyakSEG-strong-monotone-stability}, provided that $(1-A) \alpha_k \mu \in [0,1]$.
This holds because
\begin{align}
    \mu\sqnorm{x_k-\hx_k}
    &\leq
    \inprod{
        F_{\mathcal S_k}(x_k)-F_{\mathcal S_k}(\hx_k)
    }{
        x_k-\hx_k
    }
    \nonumber \\
    &\leq
    \norm{
        F_{\mathcal S_k}(x_k)-F_{\mathcal S_k}(\hx_k)
    }
    \norm{x_k-\hx_k}
    \leq
    A\norm{F_{\mathcal S_k}(x_k)}
    \norm{x_k-\hx_k} 
    \label{eqn:DecPolyakSEG-strong-monotone-and-critical-condition}
\end{align}
by strong monotonicity of $F_{\cS_k}$ and the critical condition, so substituting $x_k-\hx_k = \gamma_k F_{\mathcal S_k}(x_k)$ yields $\mu \gamma_k \le A$ and \cref{lemma:PolyakSEG-LS-key-properties}(c) gives $\alpha_k\leq\frac{\gamma_k}{1-A}$.

Finally, by \eqref{eqn:DecPolyakSEG-strong-monotone-and-critical-condition} and $L$-Lipschitzness of $F_{\cS_k}$ we have
\begin{align*}
    \norm{\hx_k-x_*}
    &\leq
    \norm{x_k-x_*}+\norm{x_k-\hx_k}
    \\
    &\leq
    \norm{x_k-x_*}
    +\frac{A}{\mu}\norm{F_{\mathcal S_k}(x_k)}
    \\
    &\leq
    \norm{x_k-x_*}
    +\frac{A}{\mu} \left( \norm{F_{\mathcal S_k}(x_k) - F_{\cS_k}(x_*)} + \norm{F_{\cS_k}(x_*)} \right)
    \\
    &\leq
    \left(1+\frac{AL}{\mu}\right)\norm{x_k-x_*}
    +\frac{A}{\mu}\norm{F_{\mathcal S_k}(x_*)}
    \\
    &\leq
    \left(1+\frac{AL}{\mu}\right)\norm{x_k-x_*}
    +
    \frac{A}{\mu}
    \underset{1\le i \le n}{\max} \norm{F_i(x_*)}
\end{align*}
which shows that $\hx_k$ also stays bounded almost surely.
\end{proof}

\paragraph{Remark.} While \cref{proposition:DecPolyakSEG_strong_monotone} uses the fact that $F = \frac{1}{n} \sum_{i=1}^n F_i$ has a finite-sum structure, all the other results of this section hold for stochastic problems given in general expectation form  $F = \mathbb{E}_\xi [F_\xi]$ under the mild assumption $\mathbb{E}_\xi \left[ \sqnorm{F_\xi (x_*)} \right] < \infty$.

\section{Numerical Experiments}
\label{sec:numerical_experiment}

We provide numerical evaluations to illustrate the two main algorithmic consequences of our theory.
First, the Polyak correction can improve over the extragradient update using equal step-sizes, even for generalized Lipschitz and H\"older continuous problems in the deterministic setting.
Second, the line-search procedure can effectively find $\gamma_k$ satisfying the critical condition and provides an empirically competitive parameter-free algorithm.
In the stochastic setting, we compare \algname{PolyakEG-LS} and \algname{DecPolyakEG-LS} with existing adaptive algorithms for stochastic monotone inclusion/variational inequality problems, respectively, on interpolated and non-interpolated problems.
This experimentally demonstrates that the effectiveness of the Polyak idea extends beyond the deterministic setting.
For performance plots, we plot the relative errors, either $\frac{\sqnorm{x_k-x_*}}{\sqnorm{x_0-x_*}}$ or $\frac{\sqnorm{F(\overline{x}_k)}}{\sqnorm{F(x_0)}}$, where $\overline{x}_k$ is the ergodic average of extrapolated points,
versus the number of operator evaluations, including rejected line-search trials.
In all experiments, the standard \algname{EG} or \algname{SEG} uses the same update step-size as the extrapolation step-size, i.e., $\alpha_k = \gamma_k$.

\subsection{Deterministic Setting}

\paragraph{Accelerating effect of the Polyak-type update step.}\mbox{}\par
\noindent
\begin{minipage}[t]{.57\textwidth}
\vspace{0pt}
We consider the 
quadratic min-max game
\begin{align}
\label{eqn:min-max-game-for-trajectory-plot}
    \underset{y \in \reals}{\text{minimize}} \,\, \underset{z \in \reals}{\text{maximize}} \,\, g(y,z) =\tfrac12y^2+5yz-25z^2 ,
\end{align}
equivalent to \eqref{eq:root_finding} with $F(y,z) =(y+5z,-5y+50z)$, which has the unique solution $(0,0)$.
We plot in \Cref{fig:2Dcontour} the first 10 iterations of \algname{EG} and \algname{PolyakEG} with $(y_0, z_0) = (1, 1)$, both using the extrapolation step-size $\gamma_k = \frac{1}{L}$ with tight Lipschitz constant $L$. 
We observe that \algname{PolyakEG} makes much larger progress per iteration compared to \algname{EG} using $\alpha_k = \gamma_k$, which can be attributed to the optimized choice of $\alpha_k$ in \algname{PolyakEG}, which is the step-size used for making the updates.

\end{minipage}
\hfill
\begin{minipage}[t]{.38\textwidth}
    \vspace{0pt}
    \centering
    \includegraphics[width=\linewidth]{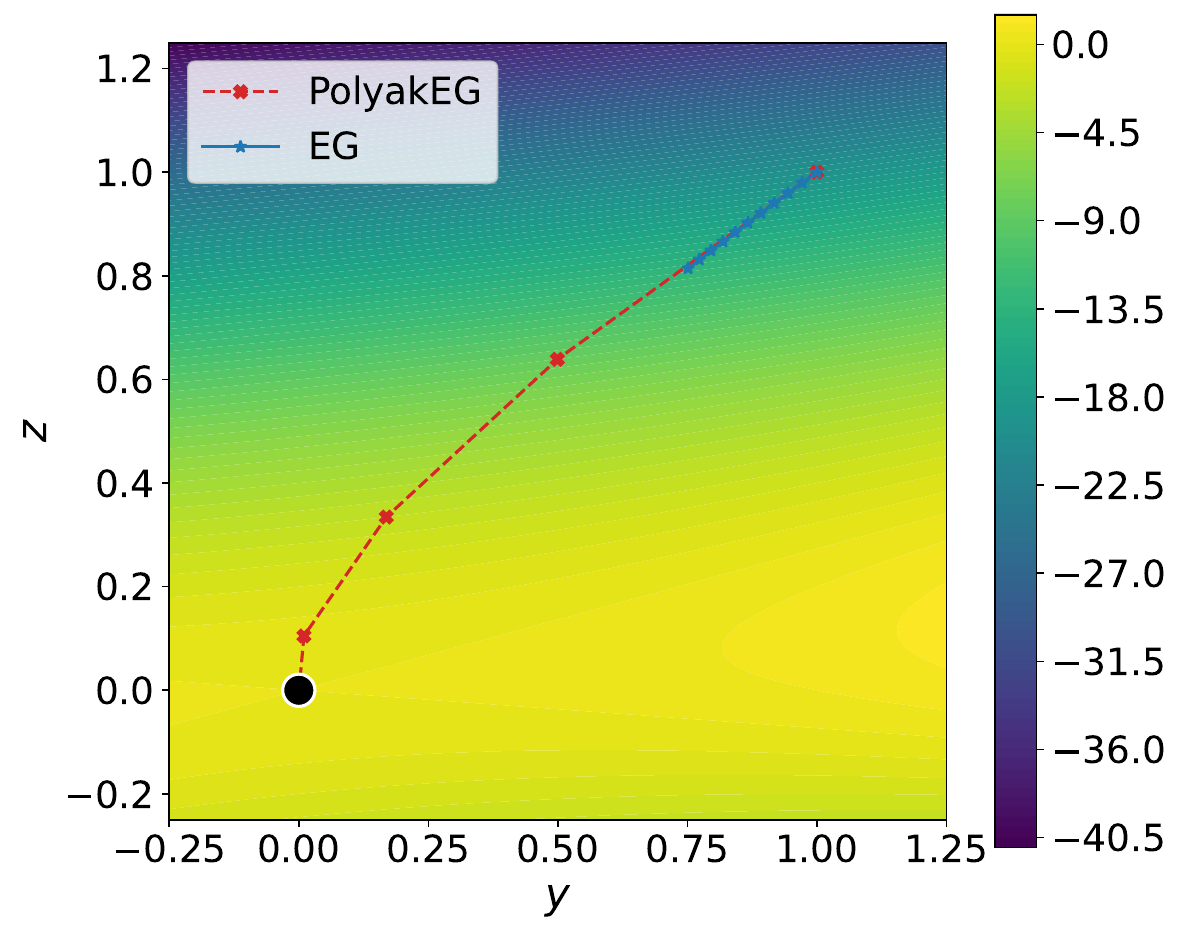}
    \captionsetup{font=small}
    \captionof{figure}{Trajectories of \algname{EG} and \algname{PolyakEG} for \eqref{eqn:min-max-game-for-trajectory-plot} using $\gamma_k=1/L$. The unique solution $(0,0)$ is marked in black.}
    \label{fig:2Dcontour}
\end{minipage}

\begin{figure}[ht]
    \centering
    \begin{subfigure}[t]{.48\textwidth}
        \centering
        \includegraphics[width=\linewidth]{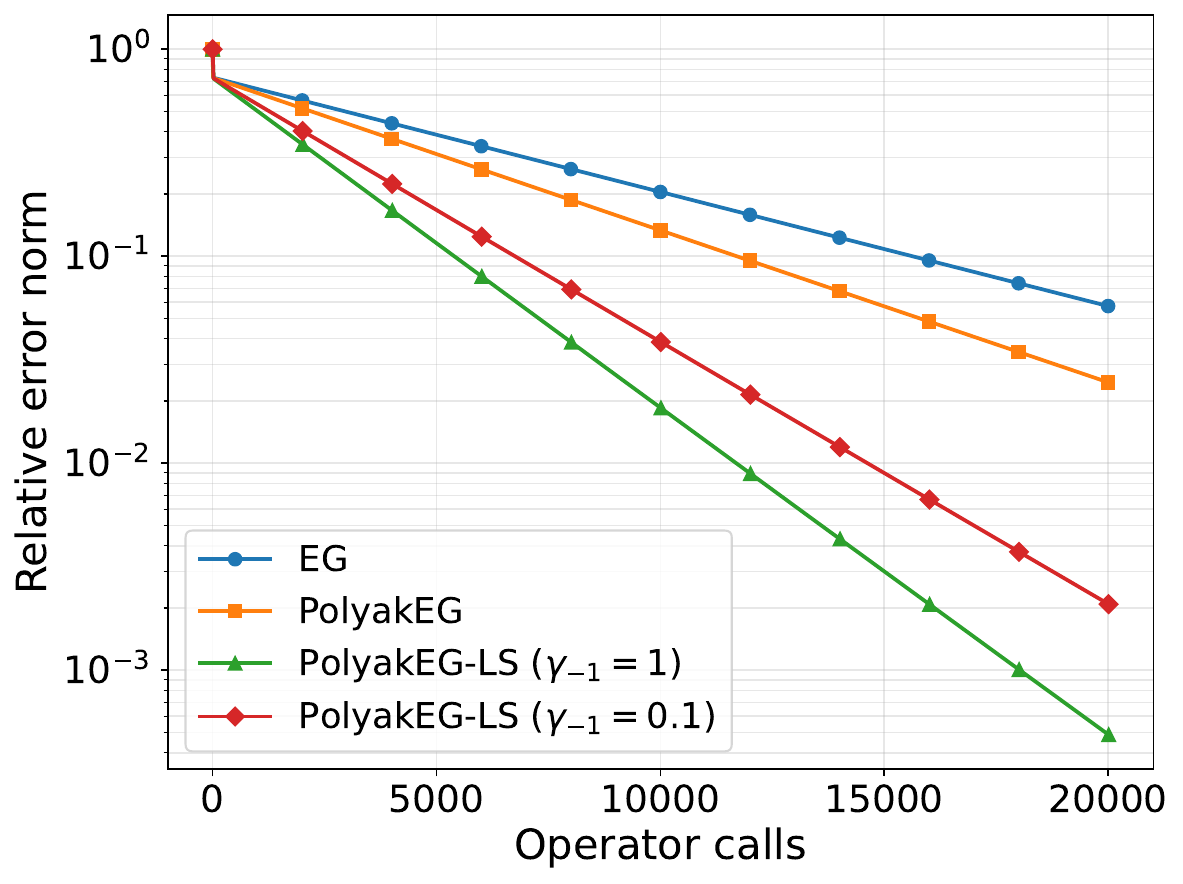}
        \caption{Affine root finding}
        \label{fig:deterministic-affine}
    \end{subfigure}
    \hfill
    \begin{subfigure}[t]{.48\textwidth}
        \centering
        \includegraphics[width=\linewidth]{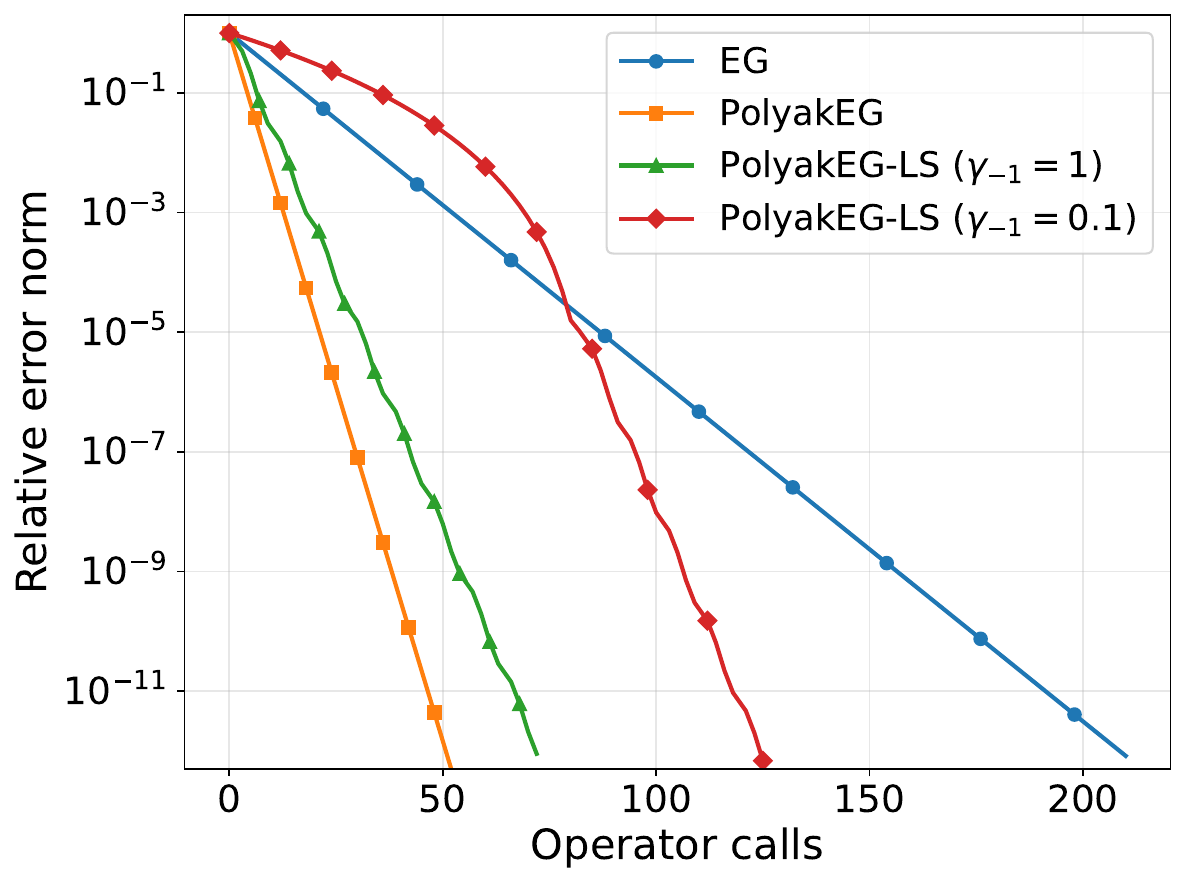}
        \caption{H\"older-continuous root finding}
        \label{fig:holder-continuous-root}
    \end{subfigure}
    \caption{Performance on deterministic, Lipschitz/H\"older-continuous problems versus operator evaluations, including rejected line-search trials. 
    Both plots report the relative error $\frac{\norm{x_k-x_*}}{\norm{x_0-x_*}}$.
    }
    \label{fig:deterministic-performance}
\end{figure}

\paragraph{Deterministic affine root finding.}
We consider a 20-dimensional strongly monotone affine operator $F(x)=Mx+b$, obtained by averaging 100 random block operators of the form $M_i = \begin{bmatrix} A_i & B_i\\ -B_i & C_i \end{bmatrix}$,
where eigenvalues of positive definite $A_i, C_i$ are uniformly distributed on $[10^3, 10^4]$ and $[1, 10]$ respectively, $B_i$ is symmetric with uniformly random eigenvalues in $[0, 100]$, and $b$ has normally distributed entries. 
We consider the theoretically motivated step-sizes: $\gamma_k = \alpha_k = \frac{1}{4L}$ for \algname{EG} and $\gamma_k = \frac{1}{3L}$ for \algname{PolaykEG}, where $L$ is the tight Lipschitz constant.
\algname{PolyakEG-LS} selects $\gamma_k$ by line search in \cref{alg:PolyakEG_linesearch_unified}, using $\beta=3$, $A=0.8$, $\lambda_{-1}^1 = 0$ and $\lambda_{-1}^0 \in \left\{ \nu_A, 10\nu_A \right\}$ (so that initial extrapolation step-size is either $1.0$ or $0.1$).
We plot the algorithms' performance starting at a Gaussian random initial point in \cref{fig:deterministic-affine}.
The fixed-step \algname{PolyakEG} converges faster than \algname{EG}, and both line-search methods achieve substantially smaller errors than versions with fixed $\gamma_k$, displaying little backtracking overhead.

\paragraph{H\"older-continuous nongradient operator.}
We consider a test problem from \citet{zhang2026convergence}: for $x\in\R^{2000}$, define
\[
    F(x)=
    \begin{cases}
        \norm{x}^{\holderexponent-1}(I+B)x & \text{if } x\ne 0\\
        0 & \text{if } x=0
    \end{cases} 
    \quad \text{with} \quad 
    B=\rho
    \begin{bmatrix}
        0 & I_{1000}\\
        -I_{1000}&0
    \end{bmatrix},
\]
where we use $\holderexponent=0.8$ and $\rho=0.1$. 
This operator has the unique root $x_*=0$ and is monotone since $\rho < \frac{2\sqrt{\holderexponent}}{1-\holderexponent}$. 
On the other hand, it is not the gradient of a function because $B$ is nonzero and skew-symmetric.
$F$ is globally $(L,\holderexponent)$-H\"older with $L=2^{1-\holderexponent}\sqrt{1+\rho^2}$.
Starting from a normalized Gaussian vector, we run both \algname{EG} and
\algname{PolyakEG} with 
\[
    \gamma_k
    =\left(\frac{A}{L}\right)^{\frac{1}{\holderexponent}}
      \norm{F(x_k)}^{\frac{1-\holderexponent}{\holderexponent}} ,
    \qquad A=0.95 .
\]
\algname{PolyakEG-LS} instead uses $\beta=3$, $\lambda_{-1}^1=0$, and
$\gamma_{-1} \in \left\{1.0, 0.1 \right\}$, without access to $L$ or $\holderexponent$.
As shown in \cref{fig:holder-continuous-root}, \algname{PolyakEG} converges approximately four times faster than \algname{EG} sharing the extrapolation step-size $\gamma_k$, while \algname{PolyakEG-LS} converges with slightly slower but competitive rates.
This illustrates that the Polyak correction can provide a clear gain for nonsmooth monotone problems, and the parameter-free version retains most of this improvement.

\begin{figure}[t]
    \centering
    \begin{subfigure}[t]{.52\textwidth}
        \centering
        \includegraphics[width=\linewidth]{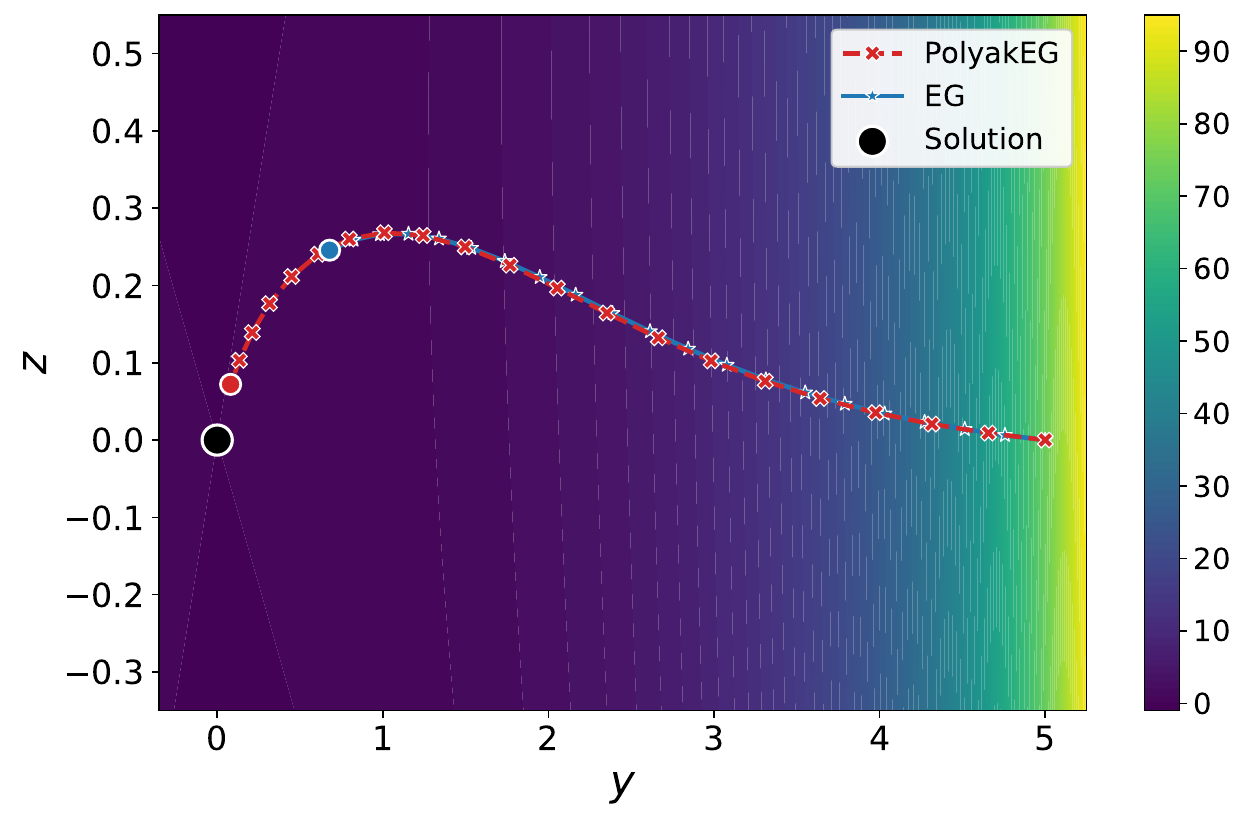}
        \caption{2D nonlinear $(L_0, L_1)$-Lipschitz game}
        \label{fig:nonlinear-sinh}
    \end{subfigure}
    \hfill
    \begin{subfigure}[t]{.46\textwidth}
        \centering
        \includegraphics[width=\linewidth]{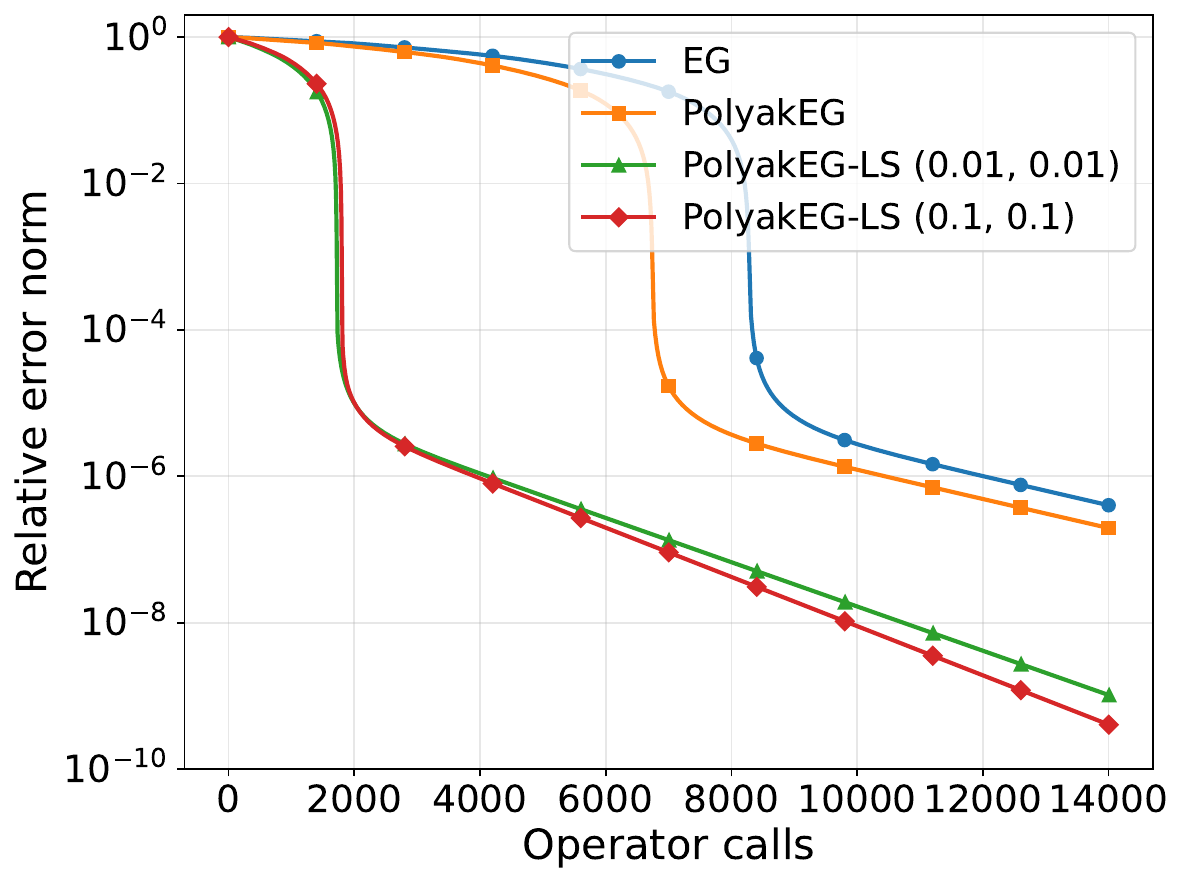}
        \caption{High-dimensional $(L_0, L_1)$-Lipschitz game}
        \label{fig:nonlinear-cosh-high-dimensional}
    \end{subfigure}
    \caption{Trajectory and performance (in terms of relative error) on $(L_0, L_1)$-Lipschitz problems.}
    \label{fig:placeholder}
\end{figure}

\paragraph{$(L_0, L_1)$-Lipschitz saddle problem on 2D.}
Consider the minimax problem
\[
    \underset{y\in\R}{\text{minimize}} \,\, \underset{z\in\R}{\text{maximize}} \quad g(y,z)=\cosh(y)+\rho yz-\cosh(z)
\]
whose saddle gradient operator is
\[
    F(y,z) = \left( \sinh(y)+\rho z, \sinh(z)-\rho y \right).
\]
The operator is strongly monotone and not globally Lipschitz continuous.
However, it is $(L_0, L_1)$-Lipschitz with $L_0 = \sqrt{1 + \rho^2}$ and $L_1 = 1 + \rho$.
\Cref{fig:nonlinear-sinh}, plotting the first 20 iterations of \algname{EG} and \algname{PolyakEG} using the common extrapolation step-size $\gamma_k = \frac{1}{L_0 + L_1 \norm{F(y_k, z_k)}}$, shows \algname{PolyakEG} trajectory more quickly reaching the unique solution $(y_*, z_*) = (0, 0)$.
Indeed, the values of $\frac{\norm{(y_{20}, z_{20})}}{\norm{(y_0, z_0)}}$ for \algname{EG} and \algname{PolyakEG} are respectively $0.144$ and $0.022$.
This illustrates the effectiveness of Polyak-type correction beyond globally Lipschitz problems.

\paragraph{$(L_0, L_1)$-Lipschitz saddle problem on higher dimensions.}
Next, we consider the following high-dimensional extension of the previous 2D $(L_0, L_1)$-Lipschitz problem:
\[
    \underset{y\in\R^{100}}{\text{minimize}} \,\, \underset{z\in\R^{100}}{\text{maximize}} \quad
    g(y,z) = \sum_{i=1}^{100}a_y\cosh\!\left(\frac{u_i^\top y}{s_i^y}\right)
    +y^\top Cz
    -\sum_{i=1}^{100}a_z\cosh\!\left(\frac{v_i^\top z}{s_i^z}\right).
\]
We take $a_y=1, a_z=1.44$ and $\rho=0.05$.
Here $s_i^y$ are geometrically spaced between 0.9 and 1.1 for $i=1,\dots,40$, $\left( s_i^y \right)^{-2}$ are linearly spaced between $2\times 10^{-3}$ and $2\times 10^{-1}$ for $i=41,\dots,100$, and $s_i^z = \sqrt{\frac{a_z}{a_y}} s_i^y$ for $i=1,\dots,100$.
We take $C = U \, \mathrm{diag}(\rho d_1,\ldots,\rho d_{100}) \, V^\top$, where $d_i=\frac{a_y}{(s_i^y)^2} = \frac{a_z}{(s_i^z)^2}$.
The saddle operator is strongly monotone with parameter $\min_i d_i = 2\times 10^{-3}$, not globally Lipschitz, but is $(L_0,L_1)$-Lipschitz with $L_0 = \max_i d_i\sqrt{1+\rho^2}$ and $L_1 = \frac{1}{(1-\rho)\min_i\{s_i^y,s_i^z\}}$.
We take $\rho = 0.3$, $A = \frac{1}{\sqrt{2}}$, $\gamma_k = \frac{1}{L_0 + L_1 \norm{F(x_k)}}$ for \algname{EG} and \algname{PolyakEG} where $x_k = (y_k, z_k)$, $\beta = 3$ and $(\lambda_{-1}^0, \lambda_{-1}^1) \in \left\{ (0.1,0.1), (0.01,0.01) \right\}$ for \algname{PolyakEG-LS}.
\Cref{fig:nonlinear-cosh-high-dimensional} shows that \algname{PolyakEG} with $\gamma_k = \frac{1}{L_0 + L_1 \norm{F(x_k)}}$ displays faster convergence compared to \algname{EG} with the same $\gamma_k$.
Both line-search variants make substantially faster progress, illustrating that the choice $\gamma_k = \frac{1}{L_0 + L_1 \norm{F(x_k)}}$ can be conservative, while \algname{PolyakEG-LS} can adapt to the local geometry of the problem to achieve rapid convergence.
\medskip

\subsection{Stochastic Setting}

\begin{figure}[t]
    \centering
    \begin{subfigure}[t]{.48\textwidth}
        \centering
        \includegraphics[width=\linewidth]{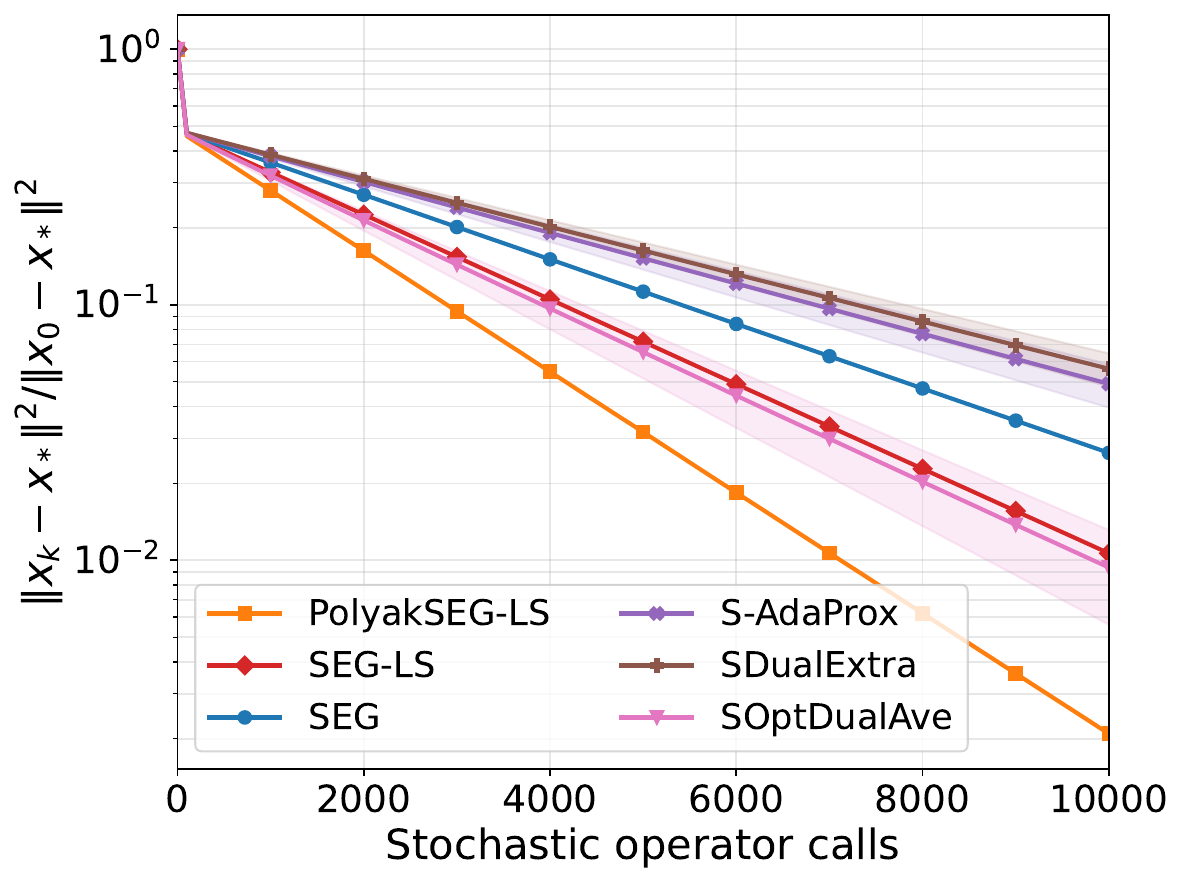}
        \caption{Interpolated stochastic affine root finding}
        \label{fig:stochastic-interpolation}
    \end{subfigure}
    \hfill
    \begin{subfigure}[t]{.48\textwidth}
        \centering
        \includegraphics[width=\linewidth]{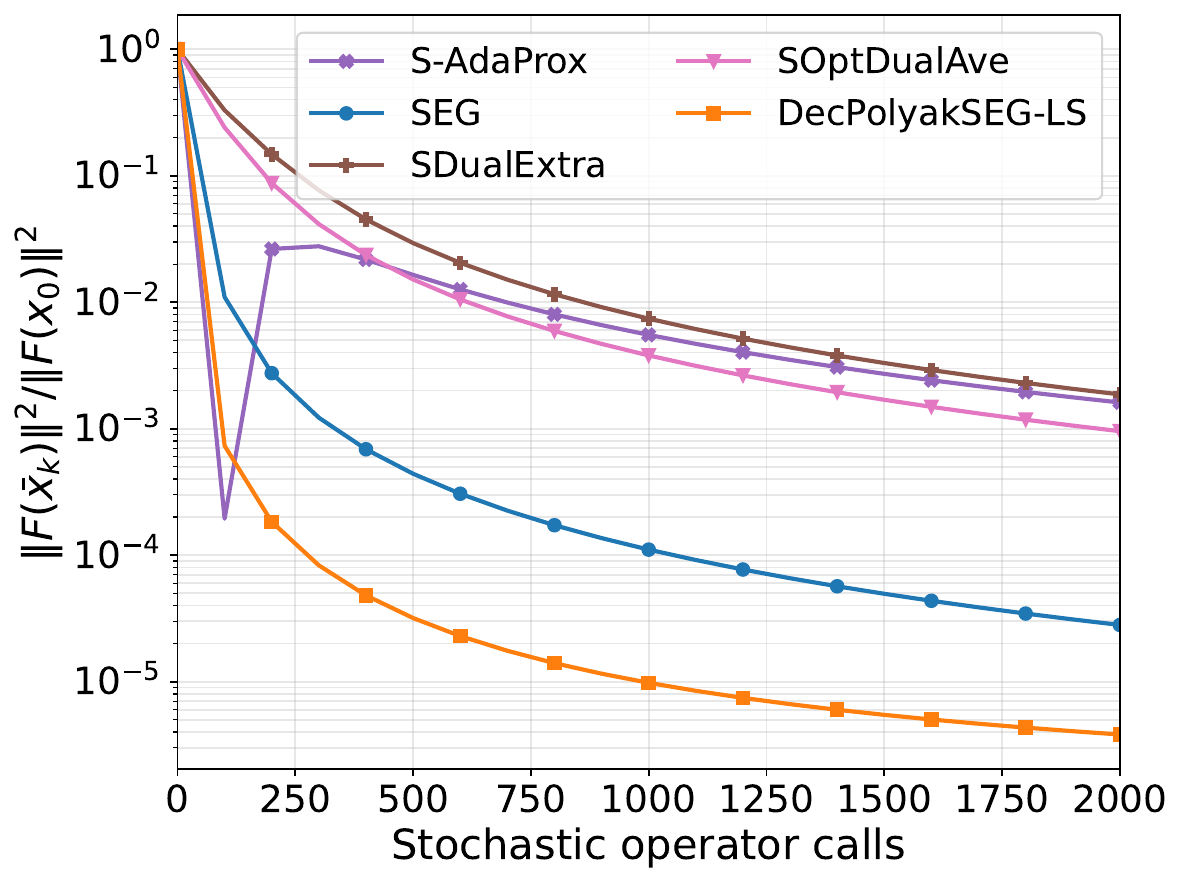}
        \caption{Stochastic robust least squares}
        \label{fig:stochastic-robust-least-squares}
    \end{subfigure}
    \caption{Performance plots for stochastic setting. Curves and shaded regions denote the mean
    and one standard deviation over (a) five trials and (b) ten trials.
    We report the relative squared distance to $x_*$ for (a) where we have linear convergence, while we  use the relative squared residual norm $\sqnorm{F(\overline{x}_k)}$ for (b) where convergence is sublinear.}
    \label{fig:stochastic-performance}
\end{figure}

\paragraph{Interpolated stochastic affine root finding.}
We consider a stochastic counterpart of the affine experiment associated with \cref{fig:deterministic-affine}, where the objective operator is $F(x)=\frac{1}{n}\sum_{i=1}^n F_i(x)$ with $F_i(x) = M_i x + b_i$, where the matrices $M_i = \begin{bmatrix} A_i & B_i\\ -B_i & C_i \end{bmatrix}$ have the same block form as in the deterministic problem and $A_i, B_i, C_i \in \mathbb{S}^{10}$.
We take $n=100$, and generate $A_i$, $B_i$ and $C_i$ to have uniformly random eigenvalues in $[10^3,10^4]$, $[0,100]$ and $[1,10]$, respectively.
A common solution $x_*$ is drawn from a normal distribution, and we set $b_i=-M_i x_*$ for $i=1,\dots,n$ so that every component satisfies $F_i(x_*)=0$. 
Each stochastic sample operator evaluation uses a mini-batch of five components.

We consider the following baseline algorithms: \algname{SEG} with constant step-size $\gamma_k = \frac{1}{2L_\mathrm{max}}$ where $L_\mathrm{max} = \max_{i=1,\dots,n} \|M_i\|_2$ is the largest component Lipschitz constant; \algname{SEG-LS}, which is the stochastic \algname{EG} using the same line-search scheme as \algname{DecPolyakSEG-LS} as considered in \citet{vaswaniPainlessStochasticGradient2019}; \algname{S-AdaProx}, which is the independent-sampling stochastic variant of \algname{AdaProx} from \citet{antonakopoulosAdaptiveExtragradientMethods2021}; and Stochastic Dual Extrapolation (\algname{SDualExtra}) and Stochastic Optimistic Dual Averaging (\algname{SOptDualAve}) algorithms from \citet{antonakopoulosSiftingNoiseUniversal2021}.

For all line-search variants, we use $\gamma_{-1} = 1$, $A = \frac{1}{\sqrt{2}}$ and line-search contraction factor $0.3$ (equivalent to $\beta = \frac{10}{3}$).
For adaptive baaseline algorithms, we use $\gamma_k = (L_\mathrm{max}^{-2} + D_k)^{-1/2}$ where $D_0 = 0$ and $D_k$ accumulates the second moment of operator evaluations or their differences, depending on the algorithm design.
We observe that, as shown in \cref{fig:stochastic-interpolation}, \algname{PolyakSEG-LS} attains the smallest error within $10^4$ oracle-call budget.
Among the remaining baseline algorithms, \algname{SOptDualAve} performs best and \algname{SEG-LS} closely matches its performance, while \algname{S-AdaProx} and \algname{SDualExtra} make slower but steady progress.

\paragraph{Stochastic robust least squares.}
We next consider the robust least-squares problem \citep{elghaouiRobustSolutionsLeastsquares1997,YangKiyavashHe2020_global}
\begin{equation}
    \label{eq:robust-least-squares-experiment}
    \underset{v\in\R^{10}}{\text{minimize}} \,\, \underset{y\in\R^{442}}{\text{maximize}} \quad 
    \norm{\mathbf{A}v-y}^2-\lambda\norm{y-y_0}^2 
\end{equation}
under stochastic operator oracle.
We use the standardized \textit{diabetes} design matrix $\mathbf{A}\in\R^{442\times10}$ from \texttt{scikit-learn} and set $\lambda=100$. 
We draw $v_*\sim\mathcal N(0,I_{10})$ and $\varepsilon\sim\mathcal N(0,I_{442})$ and set $y_0=\mathbf{A}v_*+\varepsilon$. 
The stochastic oracle is $F_\xi(x)=F(x)+\xi$, where $x=(v,y)$, $F(x)$ is the saddle operator, and $\xi\sim\mathcal N(0,I_{452})$.

All methods use the initial point $(v_0, y_0) = (0, 0)$, and initial extrapolation step-size $10^{-2}$.
We compare \algname{DecPolyakSEG-LS} with $c_{-1}=1$ and $c_k=\sqrt{k+1}$ against decreasing-step \algname{SEG} with $\gamma_k = \alpha_k = \frac{10^{-2}}{\sqrt{k+1}}$, \algname{S-AdaProx}, \algname{SDualExtra}, and \algname{SOptDualAve}. 
Following the respective theories, we take $\overline{x}_k$ as uniform average of $\hx_k$ all algorithms except for \algname{S-AdaProx}, while for \algname{S-AdaProx} we take the average weighted by step-sizes.
We observe that \algname{DecPolyakSEG-LS} attains the smallest residual, while retaining the advantage that it can be run without the knowledge of the operator's Lipschitz constant.

\section{Conclusion}
\label{sec:conclusion}

We develop a theory of monotone root-finding problem based on Polyak's principle.
Our results show that the idea of Polyak-type step-size correction extends beyond the minimization setting where the knowledge of optimal value is required, specifically to root-finding problems.
Analysis of \algname{PolyakEG} in the deterministic case separates two complementary roles
of adaptivity.
The extrapolation step $\gamma_k$ controls local operator variation so that the update step can make a sufficient progress, while the Polyak update step-size $\alpha_k$ optimizes the guaranteed progress.
Regularity assumptions quantify the range of admissible $\gamma_k$, but do not directly affect how $\alpha_k$ is chosen---this separation explains why we are able to obtain a unified convergence theorem and line-search variant \algname{PolyakEG-LS} that accommodates all of Lipschitz, H\"older-continuous, and $(L_0,L_1)$-Lipschitz operators.

We explore extensions of \algname{PolyakEG} to stochastic settings and their limitations, and they raise concrete questions for future work.
Can the deterministic algorithm's adaptivity to H\"older continuity and $(L_0,L_1)$-Lipschitzness be extended to stochastic operators?
Which stability mechanisms can be incorporated into \algname{DecPolyakSEG} to guarantee convergence for any stochastic monotone root-finding problems with adequate regularity assumptions?
We believe that these questions will lead to a direction beyond simply transferring step-size formulas or proof techniques: developing a broader understanding of Polyak-type correction that can adapt to distinct operator geometry and stochastic noise.

\section*{Acknowledgments}
The authors' contribution to this work was supported by NSF CCF 2504626 and NSF CAREER 2542902.

\bibliography{bibfile}

\end{document}